\documentclass{article}

\usepackage[usenames]{color}

\usepackage{amsmath}
\usepackage{amssymb}
\usepackage{stmaryrd}
\usepackage[matrix, arrow, curve]{xy}
\usepackage{eucal}
\usepackage{amsthm}
\usepackage{graphicx}
\usepackage{hyperref}
\hypersetup{
  colorlinks   = true, 
  urlcolor     = blue, 
  linkcolor    = blue, 
  citecolor   = red 
}
\usepackage{tikz}
\usetikzlibrary{patterns}
\newcommand{\tdot}{circle [radius=0.1]}

\newtheorem{theorem}{Theorem}[section]
\newtheorem{proposition}[theorem]{Proposition}
\newtheorem{lemma}[theorem]{Lemma}
\newtheorem{corollary}[theorem]{Corollary}

\theoremstyle{definition}
\newtheorem{definition}[theorem]{Definition}
\newtheorem{example}[theorem]{Example}
\newtheorem{remark}[theorem]{Remark}

\numberwithin{figure}{section}

\newcommand{\A}{\mathcal{A}}

\newcommand{\C}{\mathcal{C}}
\newcommand{\D}{\mathcal{D}}

\newcommand{\I}{\mathcal{I}}

\renewcommand{\L}{\mathcal{L}}
\newcommand{\M}{\mathcal{M}}

\renewcommand{\O}{\mathcal{O}}
\renewcommand{\P}{\mathcal{P}}

\newcommand{\R}{\mathcal{R}}
\renewcommand{\S}{\mathrm{Sp}} 

\renewcommand{\AA}{\mathbb{A}}

\newcommand{\CC}{\mathbb{C}}

\newcommand{\EE}{\mathbb{E}}

\newcommand{\II}{\mathbb{I}}

\newcommand{\LL}{\mathbb{L}}

\renewcommand{\SS}{\mathbb{S}}
\newcommand{\TT}{\mathbb{T}}

\newcommand{\ZZ}{\mathbb{Z}}

\newcommand{\id}{\text{id}}
\renewcommand{\i}{\infty}

\newcommand{\too}{\longrightarrow}

\newcommand{\op}{\mathrm{op}}
\newcommand{\cg}{\mathrm{cg}}
\newcommand{\prop}{\mathrm{prop}}

\newcommand{\nrv}{\mathrm{N}}

\newcommand{\PrL}{\mathcal{P}r^{\mathrm{L}}}
\newcommand{\co}{\colon\thinspace}
\newcommand{\into}{\hookrightarrow}

\DeclareMathOperator{\Mod}{Mod}

\DeclareMathOperator{\Alg}{Alg}
\DeclareMathOperator{\CAlg}{CAlg}
\DeclareMathOperator{\PGL}{PGL}
\DeclareMathOperator{\Ring}{Ring}

\DeclareMathOperator{\Br}{Br}

\DeclareMathOperator{\Aut}{Aut}
\DeclareMathOperator{\Az}{Az}

\DeclareMathOperator{\Cat}{Cat}

\DeclareMathOperator{\End}{End}
\DeclareMathOperator{\Ext}{Ext}

\DeclareMathOperator{\Fun}{Fun}

\DeclareMathOperator{\Hom}{Hom}

\DeclareMathOperator{\Map}{Map}
\DeclareMathOperator{\Mat}{Mat}

\DeclareMathOperator{\ab}{Ab}

\DeclareMathOperator{\colim}{colim}

\DeclareMathOperator{\Pic}{Pic}

\DeclareMathOperator{\Spec}{Spec}

\DeclareMathOperator{\GL}{GL}

\DeclareMathOperator{\Op}{Op}
\DeclareMathOperator{\Comm}{Comm}

\def\rrarrow{  \hspace{.05cm}\mbox{\,\put(0,-2){$\rightarrow$}\put(0,2){$\rightarrow$}\hspace{.45cm}}}

\def\rrrarrow{ \hspace{.05cm}\mbox{\,\put(0,-3){$\rightarrow$}\put(0,1){$\rightarrow$}\put(0,5){$\rightarrow$}\hspace{.45cm}}}

\def\rrrrarrow{\hspace{.05cm}\mbox{\,\put(0,-3.5){$\rightarrow$}\put(0,0){$\rightarrow$}\put(0,3.5){$\rightarrow$}\put(0,7){$\rightarrow$}
               \hspace{.45cm}}}

\renewcommand{\epsilon}{\varepsilon}

\DeclareMathOperator{\Tor}{Tor} 
\DeclareMathOperator{\Der}{Der}

\DeclareMathOperator{\pic}{\mathfrak{pic}}
\DeclareMathOperator{\br}{\mathfrak{br}}
\DeclareMathOperator{\Sq}{Sq}

\DeclareMathOperator{\Act}{Act}
\DeclareMathOperator{\LMod}{LMod}

\DeclareMathOperator{\Equiv}{Equiv}
\DeclareMathOperator{\Tot}{Tot}

\newcommand{\Rs}{{R_\star}}
\newcommand{\As}{A_\star}
\newcommand{\Bs}{{B_\star}}

\begin{document}

\title{Brauer groups and Galois cohomology of commutative ring spectra}

\author{David Gepner\thanks{Partially supported by NSF grant 1406529.}\,\, and Tyler Lawson\thanks{Partially supported by NSF grant 1206008.}}

\maketitle

\tableofcontents

\section{Introduction}
\label{sec:introduction}

The Brauer group of a field $F$, classifying central simple algebras over $F$, plays a critical role in class field theory. The definition was generalized by Auslander and Goldman \cite{auslander-goldman-brauer} to the case of a commutative ring: the Brauer group of $R$ consists of Morita equivalence classes of Azumaya algebras over $R$.

In recent years these concepts have been extended to derived algebraic geometry \cite{toen-azumaya}, to homotopy theory \cite{baker-richter-szymik}, to more general categorical frameworks \cite{johnson-azumaya-bicategory}, and generalized to the Morita theory of $E_n$-algebras \cite{haugseng-highermorita}. Associated to a commutative ring spectrum $R$, there is a category of Azumaya algebras over $R$ and a Brauer space $\Br(R)$ classifying Morita equivalence classes of such $R$. Joint work of Antieau with the first author gave an in-depth study of these Brauer spaces when $R$ is connective \cite{antieau-gepner-brauergroups}, and in particular found that the set of Morita equivalence classes could be calculated cohomologically.

There are two important tools developed in \cite{antieau-gepner-brauergroups} which make this cohomological identification possible. First, Azumaya algebras $A$ over connective $R$ are \'etale-locally trivial: there exist enough ``$\pi_*$-\'etale'' maps $R \to S$ such that $S \otimes_R A$ is Morita trivial. Second, generators descend: an $R$-linear category which is \'etale-locally a category of modules over an Azumaya algebra is a category of modules over a global Azumaya algebra. The goal of this paper is to calculate the Brauer group of {\em nonconnective} ring spectra $R$, and these tools are absent in the case when $R$ is nonconnective. Moreover, the first outright fails: there exist Azumaya algebras which are not $\pi_*$-\'etale-locally trivial.

This should not necessarily be suprising: detecting \'etale extensions on the level of $\pi_*$ is fundamentally not adequate for nonconnective ring spectra. For example, the homotopy pullback of the diagram of Eilenberg--Mac Lane spectra
\[
\xymatrix{
R \ar@{.>}[r] \ar@{.>}[d] & \CC[x,y^{\pm 1}] \ar[d]\\
\CC[x^{\pm 1},y] \ar[r] & \CC[x^{\pm 1},y^{\pm 1}]
}
\]
has a map $\CC[x,y] \to R$ which is not $\pi_*$-\'etale. On the level of module categories, however, $R$-modules are equivalent to $\CC[x,y]$-modules supported away from the origin, and so this gives an ``affine'' but nonconnective model for the open immersion $\AA^2 \setminus \{0\} \hookrightarrow \AA^2$ \cite[2.4.4]{lurie-dag8}. In this and other quasi-affine cases, the coefficient ring does not exhibit all of the useful properties of this map \cite[Section~8]{mathew-residuefields}.

Our first tool for calculations will be obstruction theory. We show that the homotopy category of those Azumaya algebras over $R$ whose underlying graded coefficient ring is a projective module over $\pi_* R$ form a category equivalent to the category of Azumaya $\pi_* R$-algebras in the graded sense (a result of Baker--Richter--Szymik \cite{baker-richter-szymik}). Moreover, we show that there exist natural exact sequences that calculate the homotopy groups of the space of automorphisms of such an Azumaya algebra. For example, the space of automorphisms of the matrix algebra $M_n(R)$ is an extension of a discrete group of ``outer automorphisms'' by a group which might be called $PGL_n(R)$. With an eye towards future applications, we have developed our obstruction theory so that one may extend from a $\ZZ$-grading to general families $\Gamma$ of elements of the Picard groupoid of $R$.

Our second tool for calculations will be descent theory. For a Galois extension of ring spectra $R \to S$ with Galois group $G$ in the sense of Rognes \cite{rognes-galois} we develop descent-theoretic methods for lifting Azumaya algebras and Morita equivalences from $S$ to $R$. In particular, there are maps $B\Pic(S)^{hG} \to \Br(S)^{hG} \xrightarrow{\sim} \Br(R)$. The first map is an equivalence above degree zero and an injection on $\pi_0$, with image consisting of those Morita equivalence classes of $R$-algebras which become Morita trivial $S$-algebras. This allows us to use calculations with the homotopy fixed-point spectrum of the Picard spectrum $\pic(S)$ from \cite{mathew-stojanoska-picard} to detect interesting Brauer classes, and employ an obstruction theory for cosimplicial spaces due to Bousfield \cite{bousfield-obstructions} to lift Azumaya algebras. In order to carry this out we need to connect the space of autoequivalences of a module to the space of autoequivalences of its endomorphism algebra. We will make heavy use of the machinery of $\i$-categories to make this possible.

In Section~\ref{sec:calculations} we will collect these together and apply them to calculations. For even-periodic ring spectra $E$, we find that the algebraic Azumaya algebras are governed by the Brauer--Wall group \cite{small-brauerwallgroup} and are generated by three phenomena: ordinary Azumaya algebras over $\pi_0 E$,  $\ZZ/2$-graded ``quaternion'' algebras over $E$, and (if $2$ is invertible) associated $1$-periodic ring spectra.  In particular, all algebraic Azumaya algebras over $KU$ are Morita trivial, and the algebraic Azumaya algebras over Lubin--Tate spectra have either $4$ or $2$ Morita equivalence classes depending on whether $2$ is invertible or not.\footnote{We note that Angeltveit--Hopkins--Lurie have announced the existence of a host of ``exotic'' elements exhausting the $K(n)$-local Brauer group of a Lubin--Tate spectrum $E$, obtained as $E$-module Thom spectra on tori.}

Finally, our most difficult calculation studies Azumaya $KO$-algebras which become Morita-trivial $KU$-algebras; we show that there exist exactly two Morita equivalence classes of these. The nontrivial Morita equivalence class is realized by an ``exotic'' $KO$-algebra lifting $M_2(KU)$ which we construct by finding a path through an obstruction theory spectral sequence. This requires a careful analysis of what happens near the bottom of the homotopy fixed-point spectral sequence for $B\Pic(KU)^{hC_2}$.

\subsection*{Acknowledgements}

The authors would like to thank 
Benjamin Antieau,
Andrew Baker,
Clark Barwick,
Saul Glasman,
Akhil Mathew,
Lennart Meier,
Birgit Richter,
Vesna Stojanoska,
and
Markus Szymik
for discussions related to this paper.

\section{Homological algebra}

In this section we will recall some important results on categories of graded objects, their algebras, and their homological algebra.

\subsection{Graded objects}

\begin{definition}
A Picard groupoid $\Gamma$ is a symmetric monoidal groupoid such that the monoidal operation makes $\pi_0(\Gamma)$ into a group.

Given a symmetric monoidal category $\C$, the Picard groupoid $\Pic(\C)$ is the groupoid of objects in $\C$ which have an inverse under the monoidal product, with maps being isomorphisms between them.
\end{definition}
We will abusively use the symbol $+$ to denote the symmetric monoidal structure on a Picard groupoid $\Gamma$, and write $0$ for the unit object.

\begin{definition}
For an ordinary category $\C$, we define the category $\C_\Gamma$ of $\Gamma$-graded objects to be the category of contravariant functors $M_\star\co \Gamma^\op \to\C$, and for $\gamma \in \Gamma$ we write $M_\gamma$ for the image. In particular, $\ab_\Gamma$ is the category of $\Gamma$-graded abelian groups.
\end{definition}

Suppose $\C$ is cocomplete and symmetric monoidal under an operation $\otimes$ with unit $\II$.  If $\otimes$ preserves colimits in each variable separately, then $\C_\Gamma$ has a symmetric monoidal closed structure given by the Day convolution product.  Specifically, its values are given by
\[
(M \otimes N)_\gamma = {\colim}_{\alpha + \beta \to \gamma} M_\alpha \otimes N_\beta,
\]
and the unit is given by the functor $\gamma \mapsto \coprod_{\Hom(\gamma, 0)} \II$.
Making choices of representatives for all isomorphism classes $[\gamma] \in \pi_0 \Gamma$ gives rise to a noncanonical isomorphism
\[
(M \otimes N)_{\gamma} \cong \coprod_{\{([\alpha], [\beta])\ |
\ \alpha + \beta \cong \gamma\}} M_\alpha \otimes_{\Aut_\Gamma(0)} N_\beta.
\]

\begin{definition}
A $\Gamma$-graded commutative ring $\Rs$ is a commutative monoid object in $\ab_\Gamma$. The unit of $\Rs$ is the induced map $\ZZ[\Aut_\Gamma(0)] \to R_0$.
\end{definition}
\begin{proposition}
The category $\Mod_{\Rs}$ of $\Gamma$-graded $\Rs$-modules is a symmetric monoidal closed abelian category, with tensor product $\otimes_{\Rs}$, internal $\Hom$ objects $F_{\Rs}(-,-)$, and arbitrary products and coproducts which are exact.
\end{proposition}

\begin{remark}
Suppose that $A$ and $G$ are abelian groups and $\epsilon$ is a bilinear pairing $A \times A \to G$. Then $\epsilon$ determines the structure of a Picard groupoid on $\Gamma = A \times BG$. The monoidal structure is split, in the sense that it is the product of the abelian group structures on $A$ and $BG$, but the natural symmetry isomorphism $\tau_{a,b}\co a + b \to b + a$ is given by $\epsilon_{a,b} \in \Aut(a+b)$. In particular, the splitting usually does not respect the symmetry isomorphism.

A $\Gamma$-graded commutative ring then consists of an $A$-indexed collection $R_\gamma$ of abelian groups, multiplication maps $R_\alpha \otimes R_\beta \to R_{\alpha+\beta}$, and a homomorphism $G \to R_0^{\times}$. These are required to satisfy associativity and unitality conditions, and the commutativity condition takes the form  $x \cdot y = \epsilon_{\alpha,\beta} (y \cdot x)$ for $x \in R_\alpha$, $y \in R_\beta$. The category of graded $R$-modules then inherits a symmetric monoidal structure using $\epsilon$ to describe a ``Koszul sign convention'' for the tensor product. We thus recover the framework of \cite{childs-garfinkel-orzech,ikai-gradedazumaya} without the assumption that $R$ is concentrated in degree zero.
\end{remark}

For $\gamma \in \Gamma$, write $\ZZ^\gamma$ for the $\Gamma$-graded abelian group obtained from the $\Gamma$-graded set $\Hom_\Gamma(-,\gamma)$ by taking the free group levelwise. We have natural isomorphisms $\ZZ^\alpha \otimes \ZZ^\beta \to \ZZ^{\alpha + \beta}$ that determine a functor $\Gamma \to \Pic(\ab_\Gamma)$.
Let the suspension operator $\Sigma^\gamma$ be the tensor product with $\ZZ^\gamma$, an automorphism of the category of $\Rs$-modules. There is an isomorphism $M_\delta \cong (\Sigma^\gamma M)_{\gamma + \delta}$, and this extends to isomorphisms
\[
M_\gamma \cong \Hom_{\Rs}(\Sigma^\gamma \Rs, M_\star).
\]
In general, if $I$ is a finite $\Gamma$-graded set, we will write $\Rs^I$ for the free $\Gamma$-graded $\Rs$-module on $I$, which can be constructed as the tensor product of $\Rs$ with the free $\Gamma$-graded abelian group on $I$.

\begin{definition}
Suppose $\As$ is an algebra in the category $\Mod_{\Rs}$. We call a right $\As$-module $P_\star$ a {\em graded generator} if $\{\Sigma^\gamma P_\star\}_{\gamma \in \Gamma}$ is a set of compact projective generators of $\Mod_{\As}$.
\end{definition}

For example, $\Rs$ is always a graded generator of $\Mod_{\Rs}$. It is unlikely to be a generator of $\Mod_{\Rs}$ in the ordinary sense unless the $\Gamma$-graded ring $\Rs$ contains units in $R_\gamma$ for each $\gamma \in \Gamma$.

Let $\theta\co \Gamma\to\Gamma'$ be a homomorphism of Picard groupoids and let $\Rs$ be a $\Gamma$-graded commutative ring.  The pullback functor $\theta^*$ from $\Gamma'$-graded modules to $\Gamma$-graded modules has a left adjoint $\theta_!$, given by left Kan extension along $\theta$.

\begin{proposition}
Suppose $\C$ is cocomplete and symmetric monoidal, and that the symmetric monoidal structure preserves colimits in each variable. Then the functor $\theta_!\co \C_\Gamma \to \C_\Gamma'$ is symmetric monoidal.
\end{proposition}

\begin{proof}
For $M$, $N$ objects of $\C_\Gamma$, we consider the square
\[
\xymatrix{
\Gamma^\op \times \Gamma^\op \ar[r]^-+ \ar[d] & \Gamma^\op \ar[d] \\
(\Gamma')^\op \times (\Gamma')^\op \ar[r]_-+ & (\Gamma')^\op.
}
\]
The object $\theta_! (M \otimes N)$ is obtained by starting with $M \otimes N\co \Gamma^\op \times \Gamma^\op \to \C$ and taking Kan extension along the two functors in the upper-right portion of the square. Because the tensor product preserves colimits in each variable, the composite Kan extension of $M \otimes N$ along the lower-left portion of the square is canonically isomorphic to $(\theta_! M) \otimes (\theta_! N)$. The natural isomorphism making the square commute determines a natural isomorphism between these two composites. Similar diagrams show that when $\theta$ preserves the unit and is compatible with the associativity, symmetry, and unit isomorphisms, $\theta_!$ does the same.
\end{proof}

In particular, the ring $\Rs$ gives rise to a $\Gamma'$-graded ring $(\theta_! R)_\star$ defined by the formula
\[
(\theta_! R)_{\gamma'}=\colim_{\gamma' \to \theta (\gamma)} R_\gamma.
\]
Moreover, an $\Rs$-module $M_\star$ determines an $(\theta_! R)_\star$-module $(\theta_! M)_\star$.

We also have the notion of a $\theta$-graded ring map $\Rs\to \Rs'$, which is just a $\Gamma'$-graded ring map $\theta_! \Rs\to \Rs'$.
Given a $\theta$-graded ring map $\Rs\to \Rs'$, we obtain a functor
\[
(-)\otimes_{\Rs} \Rs'\co \Mod_\Rs\too\Mod_{\Rs'}
\]
which sends the $\Rs$-module $M_\star$ to the $\Rs'$-module $M'_\star:=M_\star\otimes_{\Rs} \Rs'$ defined by
\[
M'=(\theta_! M)_\star\otimes_{(\theta_! R)_\star} \Rs'.
\]
Here the tensor product on the right is the usual base-change along a $\Gamma'$-graded ring map.

\begin{proposition}
For a map $\theta\co \Gamma \to \Gamma'$ and a $\theta$-graded map $\Rs \to \Rs'$, the functor
$
(-)\otimes_{\Rs} \Rs'\co \Mod_\Rs\too\Mod_{\Rs'}
$
is symmetric monoidal.
In particular, it extends to a functor
$
(-)\otimes_{\Rs} \Rs'\co \Alg_{\Rs}\too\Alg_{\Rs'}
$
between categories of algebra objects.
\end{proposition}

\begin{proposition}
Suppose $\A$ is an additive symmetric monoidal category with unit $\II$ such that the monoidal product is additive in each variable, and that we have a symmetric monoidal functor $\Gamma \to \Pic(\A)$ given by $\gamma \mapsto A^\gamma$. Then there is a canonical additive, lax symmetric monoidal functor $\phi\co \A \to \ab_\Gamma$, sending $M$ to the object $M_\star$ with
\[
M_\gamma = \Hom(A^\gamma, M).
\]
In particular, $\II_\star$ is a $\Gamma$-graded commutative ring, and $\phi$ lifts to the category of $\II_\star$-modules.
\end{proposition}

\begin{proof}
Since $\A$ is additive, the set $\Hom(M,N)$ of maps from $M$ to $N$ admits an abelian group structure such that composition is bilinear. This determines the functor $\phi$. It remains to show that $\phi$ is lax symmetric monoidal.

The lax monoidal structure map sends a pair $(A^\alpha \to M)$ in $M_\alpha$ and $(A^\beta \to N)$ in $N_\beta$ to the composite determined by
\[
A^{\alpha + \beta} \xleftarrow{\sim} A^\alpha \otimes A^\beta \to M \otimes N,
\]
an element in $(M \otimes N)_{\alpha + \beta}$. The natural associativity and commutativity diagrams
\[
\xymatrix{
A^\alpha \otimes A^\beta \ar[r]^\sim \ar[d] & A^\beta \otimes A^\alpha \ar[d] &
(A^\alpha \otimes A^\beta) \otimes A^\gamma \ar[r]^\sim \ar[d] & 
A^\alpha \otimes (A^\beta \otimes A^\gamma) \ar[d] \\
M \otimes N \ar[r]^\sim & N \otimes M &
(M \otimes N) \otimes P \ar[r]^\sim & M \otimes (N \otimes P)
}
\]
(together with a similar unitality diagram) reduce the proof that $\phi$ is a lax symmetric monoidal functor to the fact that $\Gamma \to \Pic(\A)$ is symmetric monoidal.
\end{proof}

\begin{definition}
Suppose $\A$ is an additive symmetric monoidal category such that the monoidal product is additive in each variable, and that we have a symmetric monoidal functor $\Gamma \to \Pic(\A)$ given by $\gamma \mapsto A^\gamma$. The shift operator $\Sigma^\gamma\co \A \to \A$ is defined by
\[
\Sigma^\gamma M = A^\gamma \otimes M.
\]
We then define
\[
\Hom(M,N)_\gamma:=\Hom(\Sigma^\gamma M,N).
\]
\end{definition}

The notation is compatible with the shift notation for $\Gamma$-graded abelian groups, because there is a natural isomorphism $(\Sigma^{\gamma} M)_\star  \cong \Sigma^{\gamma} (M_\star)$.

\begin{proposition}
\label{prop:graded-enrich}
In the situation of the previous definition, the $\Gamma$-graded abelian groups $\Hom(-,-)_\star$ make $\A$ into a category enriched in $\II_\star$-modules. More, this enrichment is compatible with the symmetric monoidal structure.
\end{proposition}

\begin{proof}
There are canonical isomorphisms $\Sigma^{\alpha+\beta} L \cong \Sigma^\alpha \Sigma^\beta L$.  Using this, we may define composition of graded maps by
\begin{align*}
\Hom(\Sigma^\alpha M,N)
\otimes 
\Hom(\Sigma^\beta L,M)
&\too
\Hom(\Sigma^\alpha M,N)
\otimes
\Hom(\Sigma^\alpha \Sigma^\beta L,\Sigma^\alpha M)\\
&\too\Hom(\Sigma^{\alpha + \beta} L,N).
\end{align*}
This composition is associative, and the unit $\II_\star \to \Hom(M,M)_\star$ sends $f\co A^\gamma \to \II$ to $f \otimes \id_M$.
\end{proof}

\begin{remark}
In section 14 of \cite{hovey-strickland-moravaktheory}, a group cohomology element in $H^3(\pi_0 \Gamma; \pi_1\Gamma)$ is described which obstructs our ability to make $\Gamma$-grading monoidal, in the sense of the functor $\otimes$ inducing an associative exterior product $\otimes\co \pi_\alpha(X) \otimes \pi_\beta(Y) \to \pi_{\alpha + \beta}(X \otimes Y)$. This group cohomology element is the unique $k$-invariant of the classifying space $B\Gamma$.

Since $\Gamma$ is assumed symmetric monoidal, $B\Gamma$ admits an infinite delooping and one can calculate that this $k$-invariant must vanish. This removes the obstruction to $\otimes$ inducing a monoidal pairing. However, this becomes replaced by a spectrum $k$-invariant
\[
\epsilon \in H^2(H\pi_0 \Gamma, \pi_1\Gamma) \cong \Hom(\pi_0 \Gamma, \pi_1 \Gamma)[2]
\]
which classifies the ``sign rule.''

More specifically, the sign rule is equivalent to a bilinear pairing $\pi_0 \Gamma \times \pi_0 \Gamma \to \pi_1 \Gamma$ sending $\alpha, \beta \in \pi_0 \Gamma$ to the element $\epsilon_{\alpha,\beta} \in \pi_1\Gamma$ such that for $x \in \pi_\alpha X$ and $y \in \pi_\beta Y$, $x \otimes y = \epsilon_{\alpha,\beta} (y \otimes x)$. (The elements $\epsilon_{\alpha,\beta}$ are not invariant under equivalence; the isomorphism with the group of $2$-torsion homomorphisms indicates that such Picard groupoids are determined completely by the $\epsilon_{\alpha,\alpha}$, together describing a $2$-torsion homomorphism $\pi_0 \Gamma \to \pi_1 \Gamma$ \cite{johnson-osorno}.)
\end{remark}

\begin{remark}
One needs to be extremely cautious with isomorphisms between $\Gamma$-graded objects due to the sign rule. For example, a casual expression like 
\[
F_{\Rs}(\Sigma^\alpha M_\star, \Sigma^\beta N_\star) = \Sigma^{\beta-\alpha} F_{\Rs}(M_\star, N_\star)
\]
hides several implicit isomorphisms \cite{adams-nasty}.
\end{remark}

\subsection{Graded Azumaya algebras}

We continue to fix a Picard groupoid $\Gamma$ and let $\Rs$ be a $\Gamma$-graded commutative ring with module category $\Mod_\Rs$.

\begin{definition}
If $\As$ is an algebra in $\Mod_\Rs$ with multiplication $\mu$, the opposite algebra $\As^\op$ is the algebra with the same underlying object and unit, but with multiplication $\mu \circ \tau$ precomposed with the twist isomorphism $\tau$.
\end{definition}
\begin{definition}
A $\Gamma$-graded Azumaya $\Rs$-algebra is an associative algebra $\As$ in the category $\Mod_\Rs$ such that
\begin{itemize}
\item the underlying module $\As$ is a graded projective generator of the category $\Mod_\Rs$, and
\item the natural map of algebras $\As \otimes_{\Rs} \As^\op \to \End_{\Rs}(\As)$, adjoint to the left action
\[
(\As \otimes_{\Rs} \As^\op) \otimes_{\Rs} \As \xrightarrow{1 \otimes \tau} \As \otimes_{\Rs} \As \otimes_{\Rs} \As^\op \xrightarrow{\mu (1 \otimes \mu)} \As,
\]
is an isomorphism.
\end{itemize}
\end{definition}

\begin{proposition}
If $P_\star$ is a graded generator of the category $\Mod_\Rs$, then the endomorphism algebra $\End_\Rs(P_\star)$ is an Azumaya $\Rs$-algebra.
\end{proposition}

\begin{definition}
Let $\Cat_\Rs$ be the $2$-category of additive categories which are {\em left-tensored} over the monoidal category $\Mod_\Rs$: abelian categories $\A$ with a functor $\otimes\co \Mod_\Rs \times \A \to \A$ which preserves colimits in each variable, together with a natural isomorphism
\[
\II \otimes A \xrightarrow{\sim} A
\]
and
\[
(M \otimes N) \otimes A \xrightarrow{\sim} M \otimes (N \otimes A)
\]
that respects the unit and pentagon axioms.

Morphisms in $\Cat_{\Rs}$ are $\Mod_\Rs$-linear: colimit-preserving functors $F\co A \to A'$, together with natural isomorphisms $M \otimes F(A) \to F(M \otimes A)$ that respect associativity and the unit isomorphisms. The $2$-morphisms in $\Cat_{\Rs}$ are natural isomorphisms of functors which commute with the tensor structure.
\end{definition}

\begin{remark}
In particular, a left-tensored category $\A$ inherits suspension operators by defining $\Sigma^\gamma M = (\Sigma^\gamma R) \otimes M$ via the left action. This allows us to define graded function objects by
\[
F_\A(M, N)_\gamma = \Hom_\A(\Sigma^\gamma M, N).
\]
This definition makes a $\Mod_{\Rs}$-linear category into a category enriched in $\Rs$-modules in such a way that $\Mod_{\Rs}$-linear functors preserve this enrichment.
\end{remark}

\begin{definition}
The functor
\[
\Mod\co \Alg_\Rs \to \Cat_\Rs
\]
sends an $\Rs$-algebra $\As$ to the category $\Mod_{\As}$ of right $\As$-modules in $\Mod_\Rs$, viewed as left-tensored over $\Rs$ via the tensor product in the underlying category $\Mod_\Rs$. A map $\As \to \Bs$ is sent to the functor $\Mod_{\As} \to \Mod_{\Bs}$ given by extension of scalars. (Composite ring maps have natural isomorphisms of composite functors which satisfy a coherence condition: $\Mod$ is a pseudofunctor.)
\end{definition}

The following theorems have proofs which are essentially identical to their classical counterparts; for example, see \cite{ikai-gradedazumaya}. We will sketch the main points below.

\begin{theorem}[Graded Eilenberg--Watts]
The map sending an $\As$-$\Bs$-bimodule $L_\star$ to the functor
\[
N_\star \mapsto N_\star \otimes_{\As} L_\star
\]
determines a canonical equivalence of categories from the category ${}_{\As}\Mod_{\Bs}$ of $\As$-$\Bs$-bimodules to the category of $\Mod_\Rs$-linear functors $\Mod_{\As} \to \Mod_{\Bs}$.
\end{theorem}

\begin{proof}
Functors of the form $(-) \otimes _{\As} L_\star$ are colimit-preserving and come with a natural associativity isomorphism
\[
M_\star \otimes _{\Rs} (N_\star \otimes _{\As} L_\star) \to
(M_\star \otimes _{\Rs} N_\star) \otimes _{\As} L_\star,
\]
making them maps $\Mod_{\As} \to \Mod_{\Bs}$ in $\Cat_\Rs$. This produces the desired functor. Conversely, any $\Mod_\Rs$-linear functor $G\co \Mod_{\As} \to \Mod_{\Bs}$ preserves the shift operators $\Sigma^\gamma$ and extends to a $\Gamma$-graded functor. In particular, the action map 
\[
\As \otimes_\Rs G(A_\star) \to G(\As \otimes_\Rs \As) \to G(\As)
\]
induced by the multiplication is adjoint to a ring map $\As \to F_{\Bs}(G(\As),G(\As))$ making $G(\As)$ into an $\As$-$\Bs$-bimodule. Given the canonical presentation
\[
N_\star \cong \colim \left( \bigoplus_{\Sigma^\delta \As \to \Sigma^\gamma \As \to N_\star} \Sigma^\delta \As \rightrightarrows \bigoplus_{\Sigma^\gamma \As \to N_\star} \Sigma^\gamma A_\star \right),
\]
the two colimit-preserving functors $G$ and $(-)\otimes_{\As}G(\As)$ both give us naturally isomorphic presentations
\[
G(N_\star) \cong \colim \left( \bigoplus_{\Sigma^\delta \As \to \Sigma^\gamma \As \to N_\star} \Sigma^\delta G(\As) \rightrightarrows \bigoplus_{\As \to N_\star} \Sigma^\gamma G(A_\star) \right).
\]
Therefore, these two functors are canonically equivalent.
\end{proof}

\begin{theorem}[Graded Morita theory]
Let $\As$ be an $\Rs$-algebra, and $\Mod_{\As}^{\cg}$ be the full subcategory of  $\Mod_{\As}$ spanned by the graded generators $P_\star$. Then there are canonical pullback diagrams of categories:
\[
\xymatrix{
\Pic({}_{\As}\Mod_{\As}) \ar[r] \ar[d] &
(\Mod_{\As}^{\cg})^\simeq \ar[r] \ar[d] &
\{\Mod_{\As}\} \ar[d] \\
\{\As\} \ar[r] &
(\Alg_\Rs)^\simeq \ar[r] &
\Cat_{\Rs}.
}
\]
More generally, the fiber of $(\Mod^{\mathrm{\cg}}_{\As})^\simeq \to (\Alg_\Rs)^\simeq$ over an algebra $B_\star$ is either empty or a principal torsor for the Picard group $\Pic({}_{\As}\Mod_{\As})$ of the category of bimodules.
\end{theorem}

\begin{proof}
We will first identify $\Mod^{\cg}_{\As}$ with the right-hand fiber
product. The pullback of the diagram $\Alg_{\Rs} \to \Cat_{\Rs}
\leftarrow \{\Mod_{\As}\}$ is the category of pairs $(\Bs,\phi)$,
where $\Bs$ is an $\Rs$-algebra and $\phi$ is an equivalence
$\Mod_{\Bs} \to \Mod_{\As}$ in $\Cat_{\Rs}$. Such a functor is
colimit-preserving, so by the graded Eilenberg--Watts theorem such a
functor is represented by a certain type of pair $(\Bs, P_\star)$. For
this functor to be an equivalence, the graded generator $\Bs$ must map to a graded generator $P_\star$, and we must have
$\Bs = \End_{\As}(P_\star)$. It remains to show that any such
$P_\star$ determines an equivalence of categories.

Given a right $\As$-module $P_\star$ as in the statement, we obtain an 
$\Rs$-algebra $\Bs = F_{\As}(P_\star,P_\star)$ and a functor $(-) \otimes_{\Bs} P_\star \co \Mod_{\Bs} \to \Mod_{\As}$. This functor is colimit-preserving. It also has a colimit-preserving right adjoint $F_{\As}(P_\star,-)$ because $P_\star$ is finitely generated projective.

The unit map
\[
M_\star \to F_{\As}(P_\star, P_\star \otimes_{\Bs} M_\star)
\]
is an isomorphism when $M_\star = \Sigma^\gamma B_\star$. Both sides preserve colimits, and so applying this unit to a resolution $F_1 \to F_0 \to M_\star \to 0$ where $F_i$ are (graded) free modules shows that the unit is always an isomorphism.

The counit map
\[
F_{\As}(P_\star, N_\star) \otimes_{\Bs} P_\star \to N_\star 
\]
is an isomorphism when $N_\star = P_\star$. Because the set of objects $\Sigma^\gamma P_\star$ is a set of generators there always exists a resolution $F_1 \to F_0 \to N_\star \to 0$ where $F_i$ are direct sums of shifts of $P_\star$. Again, as the functors in question preserve colimits, the counit is always an isomorphism.

We now consider the left-hand square. As pullbacks can be calculated iteratively, the pullback of a diagram $\Bs \to \Alg_{\Rs} \leftarrow (\Mod^{\cg}_{\As})^\simeq$ is equivalent to the pullback of the diagram $\{\Mod_{\Bs}\} \to \Cat_{\Rs} \leftarrow \{\Mod_{\As}\}$. If these categories are inequivalent as $\Rs$-linear categories, this is empty. If these categories are equivalent, then this is instead isomorphic to the groupoid of self-equivalences of $\Mod_{\As}$. Such a functor is given up to unique isomorphism by tensoring with an $\As$-bimodule $P_\star$, and there must exist an inverse given by tensoring with an $\As$-bimodule $Q_\star$. For these to be inverse to each other, we must have an isomorphism of $\As$-bimodules
\[
P_\star \otimes_{\As} Q_\star \cong Q_\star \otimes_{\As} P_\star \cong \As.
\]
Such a $Q_\star$ exists if and only if $P_\star$ is an invertible element in the category of bimodules.
\end{proof}

\begin{corollary}[Graded Rosenberg-Zelinsky exact sequence]
\label{cor:rz}
For an Azumaya $\Rs$-algebra $\As$, there is an exact sequence of groups
\[
1 \to (R_0)^\times \to (A_0)^\times \to \Aut_{\Alg_{\Rs}}(\As) \to \pi_0 \Pic(\Mod_\Rs).
\]
The group $\Pic(\Mod_\Rs)$ acts on the set of isomorphism classes of compact generators of $\Mod_{\As}$ with quotient the set of isomorphism classes of Azumaya $\Rs$-algebras $\Bs$ such that $\Mod_{\As} \simeq \Mod_{\Bs}$. The stabilizer of $\As$, viewed as a right $\As$-module, is the image of the outer automorphism group in $\Pic(\Mod_\Rs)$.
\end{corollary}

\begin{proof}
We consider the pullback diagram of categories
\[
\xymatrix{
\Pic({}_{\As}\Mod_{\As}) \ar[r] \ar[d] &
(\Mod_{\As}^{\cg})^\simeq \ar[d] \\
\{\As\} \ar[r] &
(\Alg_\Rs)^\simeq 
}
\]
obtained from graded Morita theory. This is a homotopy pullback
diagram of groupoids, and so we may take the nerve and obtain a long exact sequence in homotopy groups at the basepoint $\As$ of $\Pic$.  Put together, this gives an exact sequence
\[
\!\!\! 1 \to \Aut_{\Pic({}_{\As}\Mod_{\As})}(\As) \to \Aut_{\Mod_{\As}}(\As) \to \Aut_{\Alg_{\Rs}}(\As) \to \pi_0 \Pic({}_{\As}\Mod_{\As}).
\]
Moreover, the category of $\As$-bimodules is equivalent to the category of modules over $\As \otimes_{\Rs} \As^\op$, which is Morita equivalent to $\Rs$. This gives us a symmetric monoidal equivalence of categories $\Pic({}_{\As}\Mod_{\As}) \simeq \Pic({\Rs})$ that carries $\As$ to $\Rs$. The desired description of this exact sequence follows by identifying $\Aut_{\Mod_{\As}}(\As)$ with $A_0^\times$ and $\Aut_{\Pic(\Rs)}(\Rs)$ with $R_0^\times$.

Similarly, the description of the action of $\Pic$ follows by identifying this fiber square with the principal fibration associated to the map $(\Alg_{\Rs})^\simeq \to (\Cat_{\Rs})^\simeq$.
\end{proof}

\begin{remark}
In the exact sequence above, suppose $v \in A_\gamma$ is a unit in the graded ring $\As$. Then conjugation by $v$ determines an element in $\Aut_{\Alg_\Rs}(\As)$ whose image in $\Pic(\Mod_\Rs)$ is $[\Sigma^\gamma A]$.
\end{remark}

\subsection{Matrix algebras over graded commutative rings}
\begin{definition}
Let $\Rs$ be a $\Gamma$-graded commutative ring.
An $\Rs$-algebra is a matrix $\Rs$-algebra if it is isomorphic to the endomorphism $\Rs$-algebra
\[
\End_{\Rs}(M_\star) = F_{\Rs}(M_\star,M_\star)
\] of an $\Rs$-module of the form $M_\star\cong \Rs^I$ for some $\Gamma$-graded set $I$.
\end{definition}

In general we write $\Mat_{I}(\Rs)$ for the $\Gamma$-graded matrix algebra $\End_{\Rs}(\Rs^{I})$ and $\GL_{I}(\Rs)$ for the group $[\Aut_{\Rs}(\Rs^{I})]_0^\times$ of automorphisms of the graded $\Rs$-module $\Rs^{I}$.

\begin{proposition}
If $R_\gamma=0$ for $\gamma\neq 0$ then there is an isomorphism of groups
\[\GL_{I}(R)\cong\prod_{\gamma\in\Gamma}\GL_{I_\gamma}(R_0),\]
where the groups on the right are the usual general linear groups of the commutative ring $R_0$.
\end{proposition}

\begin{proposition}
If $I$ is a finite $\Gamma$-graded set, then $I$ is a disjoint union of $\Hom(\gamma_i,-)$ for some $I$, and there is a canonical isomorphism of $\Rs$-modules
\[
\End_{\Rs}(\Rs^{I}) \cong \bigoplus_{i,j \in I} \Sigma^{\gamma_j - \gamma_i} \Rs.
\]
In particular, there is a natural $\Gamma$-graded set $\partial I$ such that it is of the form $\Rs^{\partial I}$.
\end{proposition}

\begin{proposition}
The formation of matrix algebras is compatible with base-change.
That is, for any homomorphism $\theta\co \Gamma\to\Gamma'$ of abelian groups, any $\theta$-graded ring map $\Rs\to \Rs'$, and any finite $\Gamma$-graded set $I$, the canonical $\Rs'$-algebra map
\[
\Mat_{I}(\Rs)\otimes_{\Rs} \Rs'\too\Mat_{\theta_! I}(\Rs')
\]
is an isomorphism.
\label{basechange}
\end{proposition}

\begin{proof}
Write $\partial I$ for the $\Gamma$-graded set as in the previous proposition. First, let us assume that $\theta$ is the identity of $\Gamma$, so that $\Rs\to \Rs'$ is just a $\Gamma$-graded ring map. 
Then $\theta_! \partial I=\partial I$ and the map $\Rs^{\partial I} \otimes_{\Rs} \Rs'\to (\Rs')^{\partial I}$ is an equivalence between free $\Rs'$-modules on the same $\Gamma$-graded set.

Now suppose instead that $\theta$ is arbitrary and $\Rs'=\theta_! \Rs$.
Then the desired map is a composite
\[
\Rs^{\partial I} \otimes_{\Rs} \Rs' \cong\theta_!(\Rs^{\partial I})\cong(\theta_! \Rs)^{\partial\theta_! I}.
\]

Finally, an arbitrary $\theta$-graded ring map $R\to R'$ is a composite of ring maps of the type treated above, so the result follows.
\end{proof}

\subsection{Derivations and Hochschild cohomology}
\label{sec:derivations}

The following recalls some of Quillen's work on cohomology for associative rings \cite{quillen-ring-cohomology}.

In a symmetric monoidal abelian category where the monoidal operation $\otimes$ preserves colimits in each variable, any algebra $A$ sits in a short exact sequence
\[
0 \to \LL_A \to A \otimes A^\op \to A \to 0
\]
of $A$-bimodules, split (as left modules) by the unit. If $A$ is the tensor algebra on a projective object $P$, then $\LL_A$ can be identified with the projective bimodule $A \otimes P \otimes A^\op$. Moreover, for any $A$-bimodule $M$ with associated square-zero extension $M \rtimes A \to A$ in $\Alg_A$, there are canonical isomorphisms
\[
\Der(A,M) = \Hom_{\Alg_A / A}(A, M \rtimes A) \cong \Hom_{{}_A\Mod_A}(\LL_A,M).
\]
This allows us to relate the derived functors of derivations, in the sense of \cite{quillen-ring-cohomology}, to Hochschild cohomology in this category. The Andr\'e-Quillen cohomology groups of $A$ with coefficients in $M$ may be identified with the nonabelian derived functors $\Der^s(A,M)$. Applying the right derived functors of $\Hom_{{}_A\Mod_A}(-,M)$ to the exact sequence defining $\LL_A$ gives us isomorphisms
\[
\Der^s(A,M) \to HH^{s+1}(A,M)
\]
for $s > 0$ and an exact sequence
\[
0 \to HH^0(A,M) \to M \to \Der(A,M) \to HH^1(A,M) \to 0.
\]

\begin{proposition}
\label{prop:hochschild-vanishing}
Suppose $\As$ is an Azumaya $\Rs$-algebra.  For any $\As$-bimodule $M_\star$ in the category of $\Gamma$-graded $\Rs$-modules, we have a short exact sequence
\[
0 \to HH^0(\As, M_\star) \to M_\star \to \Der(\As,M_\star) \to 0.
\]
Both the Hochschild cohomology groups $HH^s_{\Rs}(\As, M_\star)$ and the derived functors $\Der^s_{\Rs}(\As, M_\star)$ vanish for $s > 0$.
\end{proposition}

\begin{proof}
Consider the short exact sequence
\[
0 \to \LL_{\As} \to \As \otimes_{\Rs} \As^\op \to \As \to 0
\]
of bimodules. The center bimodule is free, hence projective. Moreover, under the chain of Morita equivalences
\[
\Mod_{\Rs} \simeq \Mod_{\End_{\Rs}(\As)} \simeq \Mod_{\As \otimes_{\Rs} \As^\op},
\]
the image of the projective $\Rs$-module $\Rs$ is $\As$, and hence $\As$ is also projective. Therefore, the sequence splits and $\LL_{\As}$ is projective too.
\end{proof}

\section{Obstruction theory}
\label{sec:matrix-algebras}

\subsection{Gradings for ring spectra}

\begin{definition}
Let $R$ be an $\EE_\i$-ring spectrum, with $\Gamma_R$ the algebraic Picard groupoid of invertible $R$-modules and homotopy classes of equivalences; similarly let $\Gamma_\SS$ be the Picard groupoid of the sphere spectrum. A {\em grading} for $R$ is a Picard groupoid $\Gamma$ together with a commutative diagram
\[
\xymatrix{
& \Gamma \ar[dr]\\
\Gamma_{\SS} \ar[rr] \ar[ur]^\nu && \Gamma_R
}
\]
of Picard groupoids.

The {\em period} of the grading is the minimum of the set
\[
\{n > 0 \mid \nu[S^n] = 0\text{ in }\pi_0 \Gamma\} \cup \{\infty\},
\]
where $[S^n] \in \pi_0 \Gamma_{\SS}$ is the equivalence class of the $n$-sphere.
\end{definition}

A grading provides a chosen lift of the suspension $\Sigma R$ to $\Gamma$ such that the twist on $\Sigma R\otimes \Sigma R$ lifts the automorphism $-1 \in (\pi_0 R)^\times$; it also provides an action $\Gamma_\SS \times \Gamma \to \Gamma$, $(n,\gamma) \mapsto n + \gamma$, compatible with that on $\Gamma_R$. The minimal and maximal options are $\Gamma_\SS$-grading (usually referred to as ``$\ZZ$-grading'') and $\Gamma_R$-grading (usually referred to as ``Picard-grading''). If $R$ is connective (and nontrivial) then $\pi_0 \Gamma_\SS \to \pi_0 \Gamma_R$ is a monomorphism, and so $R$ has period $\i$ (usually referred to as ``not being periodic'').

Throughout this section we will assume that we have chosen a grading for $R$. This produces elements $R^\gamma \in \Mod_R$ for $\gamma \in \Gamma$ and gives the category of $R$-modules $\Gamma$-graded homotopy groups $\pi_\star M$ as in Section~\ref{sec:pic-spectra}. These homotopy groups preserve coproducts and filtered colimits, as well as take cofiber sequences to long exact sequences. The fact that weak equivalences are detected on $\ZZ$-graded homotopy groups implies the following.
\begin{proposition}
If $R$ has a grading by $\Gamma$, a map $X \to Y$ of $R$-modules is an equivalence if and only if the map $\pi_\star X \to \pi_\star Y$ is an isomorphism of $\pi_\star R$-modules.
\end{proposition}

\begin{proposition}
If $A$ is an $R$-algebra, the $\Gamma$-graded groups $\pi_\star A$ form a $\pi_\star R$-algebra. If $A$ is a commutative $R$-algebra, $\pi_\star A$ is a graded-commutative $\pi_\star R$-algebra.
\end{proposition}

\subsection{Picard-graded model structures}
\label{sec:picard-graded-model}

In this section we describe model structures on categories of $R$-modules and $R$-algebras based on using elements of $\Pic(R)$ as basic cells.  The structure of this section is based on Goerss--Hopkins' work on obstruction theory for algebras over an operad \cite{goerss-hopkins-summary}, which in turn is based on Bousfield's work \cite{bousfield-cosimplicial}.  We carry this out under the simplifying assumptions that we are not using an auxiliary homology theory, and that the operad in question is the associative operad. However, we will remove the assumption that the base category is the stable homotopy category, and allow ourselves the use of homotopy groups graded by a Picard groupoid $\Gamma$ rather than integer-graded homotopy groups.

In this section we work in the flat stable model category structure on symmetric spectra (the $S$-model structure of \cite{shipley-convenient}). Fix a commutative model for our $\EE_\i$-ring spectrum $R$, and let $\Mod_R$ be the category of $R$-modules. We also fix a grading $\Gamma$ for $R$ as in the previous section, giving any $R$-module $M$ natural $\Gamma$-graded homotopy groups $\pi_\star M$.

According to \cite[2.6-2.7]{shipley-convenient}, the category $\Mod_R$ is a cofibrantly generated, proper, stable model category with generating sets of cofibrations and acyclic cofibrations with cofibrant source; it is also, compatibly, a simplicial model category (e.g. see \cite{proeinfty} for references in this direction).  The smash product $\wedge_R$ and function object $F_R(-,-)$ give $\Mod_R$ a symmetric monoidal closed structure under which $\Mod_R$ is a monoidal model category, and the category $\Alg_R$ of associative $R$-algebras is a cofibrantly generated simplicial model category with fibrations and weak equivalences detected in $\Mod_R$ \cite{shipley-schwede-algebrasmodules}. We let $\TT$ denote the monad taking $M$ to the free $R$-algebra $\TT(M) = \bigvee M^{\wedge_R n}$; algebras over $\TT$ are associative $R$-algebras.

The following definitions are dual to those in Bousfield \cite{bousfield-cosimplicial}, taking the category $\Gamma$ as generating a class $\P$ of cogroup objects.

\begin{definition}
Let $\D_R$ denote the homotopy category of $\Mod_R$.
\begin{enumerate}
\item A map $p\co X \to Y$ in $\D_R$ is {\em $\Pic$-epi} if the map $\pi_\star X \to \pi_\star Y$ is surjective.
\item An object $A \in \D_R$ is {\em $\Pic$-projective} if the map $p_*\co [A,X] \to [A,Y]$ is surjective whenever $p\co X \to Y$ is $\Pic$-epi.
\item A morphism $A \to B$ in $\Mod_R$ is a {\em $\Pic$-projective cofibration} if it has the left lifting property with respect to all $\Pic$-epi fibrations in $\Mod_R$.
\end{enumerate}
\end{definition}

\begin{remark}
Technically speaking, we should include the group $\Gamma$ in the notation, but we do not.
\end{remark}

Any object $P \in \Pic(R)$ with a lift to an element $\gamma \in \Gamma$ is  is automatically $\Pic$-projective, and the class of projective cofibrations is closed under coproducts, suspensions, and desuspensions.  There are enough $\Pic$-projective objects: to construct a $\Pic$-projective $P$ and a map $P \to X$ inducing a surjection $\pi_\star P \to \pi_\star X$, we can choose generators $\{x_\alpha \in \pi_{\gamma_\alpha} X\}$ of $\pi_\star X$ which are represented by a map $\bigvee_\alpha R^{\gamma_\alpha} \to X$. We can then describe a model structure on the category $s\Mod_R$ of simplicial $R$-modules.

\begin{definition}
Let $f\co X_\bullet \to Y_\bullet$ be a map of simplicial $R$-modules.
\begin{enumerate}
\item The map $f$ is a {\em $\Pic$-equivalence} if the map $\pi_\gamma f\co  \pi_\gamma X_\bullet \to \pi_\gamma Y_\bullet$ is a weak equivalence of simplicial abelian groups for all $\gamma \in \Gamma$.
\item The map $f$ is a {\em $\Pic$-fibration} if it is a Reedy fibration and the map $\pi_\gamma f\co \pi_\gamma X_\bullet \to \pi_\gamma Y_\bullet$ is a fibration of simplicial abelian groups for all $\gamma \in \Gamma$.
\item The map $f$ is a {\em $\Pic$-cofibration} if the latching maps
\[
X_n \coprod_{L_n X} L_n Y \to Y_n
\]
are $\Pic$-projective cofibrations for $n \geq 0$.
\end{enumerate}
\end{definition}

\begin{theorem}[{\cite{bousfield-cosimplicial}}]
These definitions give the category $s\Mod_R$ of simplicial $R$-modules the structure of a simplicial model category, which we call the {\em $\Pic$-resolution model structure}. This model structure is cofibrantly generated, and has generating sets of cofibrations and acyclic cofibrations with cofibrant source.  The forgetful functor to simplicial $R$-modules (with the Reedy model structure) creates fibrations.
\end{theorem}

As in \cite[Section~3]{goerss-hopkins-summary}, for a simplicial $R$-module $X$ and $\gamma \in \Gamma$ we have ``natural'' homotopy groups $\pi_n^\natural(X;\gamma)$.  On geometric realization there is a homotopy spectral sequence with $E_2$-term
\[
\pi_p \pi_\gamma(X) \Rightarrow \pi_{p+\gamma} |X|\,.
\]
The $E_2$-term of this spectral sequence comes from an exact couple, the spiral exact sequence \cite[Lemma~3.9]{goerss-hopkins-summary}:
\[
\cdots \to \pi_{n-1}^\natural(X;\gamma) \to \pi_n^\natural(X;\gamma) \to \pi_n \pi_\gamma(X) \to \pi_{n-2}^\natural(X;\gamma) \to \cdots
\]

As applications of the $\Pic$-resolution model structure, we obtain $\Pic$-graded K\"unneth and universal coefficient spectral sequences.

\begin{theorem}
\label{thm:picard-kunneth-universal}
For $X, Y \in \D_R$, there are spectral sequences of $\Gamma$-graded $\Rs$-modules:
\begin{align*}
\Tor^{\Rs}_{p,\gamma}(\pi_\star X, \pi_\star Y) &\Rightarrow \pi_{p+\gamma}(X \wedge_R Y)\\
\Ext^{s,\tau}_{\Rs}(\pi_\star X, \pi_\star Y) &\Rightarrow \pi_{-s-\tau} F_R(X,Y)
\end{align*}
\end{theorem}

\begin{proof}
Lift $X$ and $Y$ to $\Mod_R$, cofibrant or fibrant as appropriate. Then choose a cofibrant replacement $P \to X$, where $X$ is viewed as a constant simplicial object in the $\Pic$-resolution model category structure.  The result is a simplicial $R$-module, augmented over $X$, such that the map $|P| \to X$ is a weak equivalence and such that the associated simplicial object $\pi_\star P$ is levelwise projective as a $\Gamma$-graded $\pi_\star R$-module.  The spectral sequences in question are associated to the geometric realization of $P \wedge_R Y$ and the totalization of $F_R(P,Y)$, which are equivalent to the derived smash $X \wedge_R Y$ and derived function object $F_R(X,Y)$, respectively.
\end{proof}

\begin{corollary}
\label{cor:freetofree}
Suppose $P$ is a cofibrant $R$-module such that $\pi_\star P$ is a projective $\pi_\star R$-module. Then $\pi_\star \TT(P)$ is isomorphic to the free $\pi_\star R$-algebra on $\pi_\star P$.
\end{corollary}

\begin{proof}
This follows by first observing that the K\"unneth formula degenerates to isomorphisms
\[
\pi_\star (P \wedge_R \cdots \wedge_R P) \cong \pi_\star P \otimes_{\pi_\star R} \cdots \otimes_{\pi_\star R} \pi_\star P,
\]
and then applying $\pi_\star$ to the identification
\[
\TT(P) \cong \bigvee_{k \geq 0} P^{\wedge_R k}
\]
of $R$-modules.
\end{proof}

The $\Pic$-resolution model structure on simplicial $R$-modules now lifts to $R$-algebras.

\begin{theorem}
There is a simplicial model category structure on $s\Alg_R$ such that the forgetful functor $s\Alg_R \to s\Mod_R$ creates weak equivalences and fibrations.  We call this the {\em $\Pic$-resolution model structure} on simplicial $R$-algebras. This model structure is cofibrantly generated, and has generating sets of cofibrations and acyclic cofibrations with cofibrant source.

For each $X \in s\Alg_R$, there is a $\Pic$-equivalence $Y \to X$ with the following properties:
\begin{enumerate}
\item The simplicial object $Y$ is cofibrant in the $\Pic$-resolution model category structure on $s\Alg_R$.
\item \cite[3.7]{goerss-hopkins-summary} There are objects $Z_n$, which are wedges of cofibrant $R$-modules in $\Pic(R)$, such that the underlying degeneracy diagram of $Y$ is of the form
\[
Y_n = \TT\left(\coprod_{\phi: [n] \twoheadrightarrow [m]} Z_m\right).
\]
\end{enumerate}
\end{theorem}

Given this structure, we can use Goerss--Hopkins' moduli tower of Postnikov approximations to produce an obstruction theory. This both classifies objects and constructs a Bousfield-Kan spectral sequence for spaces of maps between $R$-algebras using $\Gamma$-graded homotopy groups. In order to describe the resulting obstruction theories, let $\Der^s_{\Alg_{\pi_\star R}}$ denote the derived functors of derivations in the category of $\Gamma$-graded $\pi_\star R$-algebras as in section~\ref{sec:derivations}.

\begin{theorem}
\begin{enumerate}
\item There are successively defined obstructions to realizing an algebra $\As \in \Alg_{\pi_\star R}$ by an $R$-algebra $A$ in the groups
\[
\Der^{s+2}_{\Alg_{\pi_\star R}}(\As, \Omega^s \As),
\]
and obstructions to uniqueness in the groups 
\[
\Der^{s+1}_{\Alg_{\pi_\star R}}(\As, \Omega^s \As),
\]
for $s \geq 1$.
\item For $R$-algebras $X$ and $Y$, there are successively defined obstructions to realizing a map $f \in \Hom_{\Alg_{\pi_\star R}}(\pi_\star X, \pi_\star Y)$ in the groups
\[
\Der^{s+1}_{\Alg_{\pi_\star R}}(\pi_\star X, \Omega^s \pi_\star Y),
\]
and obstructions to uniqueness in the groups 
\[
\Der^{s}_{\Alg_{\pi_\star R}}(\pi_\star X, \Omega^s \pi_\star Y),
\]
for $s \geq 1$.
\item Let $\phi \in \Map_{\Alg_R}(X,Y)$ be a map of $R$-algebras.  Then there is a fringed, second quadrant spectral sequence abutting to
\[
\pi_{t-s} (\Map_{\Alg_R}(X,Y),\phi),
\]
with $E_2$-term given by
\[
E_2^{0,0} = \Hom_{\Alg_{\pi_\star R}}(\pi_\star X, \pi_\star Y)
\]
and
\[
E_2^{s,t} = \Der^s_{\Alg_{\pi_\star R}}(\pi_\star X, \Omega^t \pi_\star Y)\text{\ for }t > 0.
\]
\end{enumerate}
\end{theorem}

This theorem is obtained using simplicial resolutions. Given an $R$-algebra $A$ we form a simplicial resolution of $A$ by free $R$-algebras, which becomes a resolution of $\pi_\star A$ by free $\pi_\star R$-algebras by Corollary~\ref{cor:freetofree}. We get the spectral sequences for mapping spaces from the associated homotopy spectral sequence (see \cite{bousfield-cosimplicial}). The obstruction theory for the construction of such $A$, instead, relies on constructing {\em partial} resolutions $P_n A$ as simplicial free $R$-algebras whose homotopy spectral sequence degenerates in a specific way, and then identifying the obstruction to extending the construction of $P_n(A)$ to $P_{n+1}(A)$ as lying in an Andr\'e--Quillen cohomology group.

\subsection{Algebraic Azumaya algebras}
\label{sec:matrix-obstructions}

We now apply the obstruction theory of the previous section to the algebraic case.  We continue to let $R$ be an $\EE_\i$-ring spectrum with a grading by $\Gamma$.

We recall that an algebra $A$ is an Azumaya $R$-algebra if $A$ is a compact generator of $\D_R$ and the left-right action map $A \wedge_R A^\op \to \End_R(A)$ is an equivalence in $\D_R$ \cite{baker-richter-szymik}.

\begin{proposition}
Suppose $A$ is an $R$-algebra such that $\pi_\star A$ is a projective $\pi_\star R$-module. Then $\pi_\star A$ is an Azumaya $\pi_\star R$-algebra if and only if $A$ is an Azumaya $R$-algebra.
\end{proposition}

\begin{proof}
The projectivity of $\pi_\star A$ makes the K\"unneth and universal coefficient spectral sequences of Theorem~\ref{thm:picard-kunneth-universal} degenerate. We find that that the action map $A \wedge_R A^\op \to \End_R(A)$ becomes, on $\pi_\star$, the map $\pi_\star A \otimes_{\pi_\star R} \pi_\star A^\op \to \End_{\pi_\star R}(\pi_\star A)$, and so the two conditions are equivalent.
\end{proof}

We have a similar result about Morita triviality.
\begin{proposition}
Let $M$ be an $R$-module whose $\Gamma$-graded homotopy groups $\pi_\star M$ form a finitely generated projective $\pi_\star R$-module.
The function spectrum $\End_R(M)$ has homotopy groups given by the $\pi_\star R$-algebra $\Hom_{\pi_\star R}(\pi_\star M, \pi_\star M)$.  The center of this algebra is the image of $\pi_\star R$, and if $\pi_\star M$ is a graded generator this algebra is Morita equivalent to $\pi_\star R$ in the category of $\Gamma$-graded $\pi_\star R$-algebras.
\end{proposition}

\begin{definition}
An $R$-algebra is said to be an {\em algebraic $\Gamma$-graded Azumaya algebra} over $R$ if the multiplication on $\pi_\star A$ makes it into an Azumaya $\pi_\star R$-algebra.
\end{definition}

We may apply the Goerss--Hopkins obstruction theory to algebraic Azumaya $R$-algebras. Much of the following is originally due to Baker--Richter--Szymik \cite[6.1]{baker-richter-szymik}.

\begin{theorem}
\label{thm:obstruction-calcs}
\begin{enumerate}
\item Any Azumaya $\pi_\star R$-algebra is isomorphic to $\pi_\star A$ for some $\Gamma$-graded algebraic Azumaya $R$-algebra $A$.
\item Suppose $A$ is a $\Gamma$-graded algebraic Azumaya $R$-algebra.
For any $R$-algebra $S$ (not necessarily Azumaya), the natural map 
\[
[A, S]_{\Alg_R} \xrightarrow{\pi_\star} \Hom_{\Alg_{\pi_\star R}} (\pi_\star A, \pi_\star S),
\]
is an isomorphism. For any map $\phi\co A \to S$ of $R$-algebras (making $\pi_\star S$ into a $\pi_\star A$-bimodule) and any $t > 0$, we have an isomorphism
\[
\pi_t (\Map_{\Alg_R} (A, S), \phi) \cong (\pi_t S)/HH^0(\pi_\star A, \Omega^t \pi_\star S).
\]
\item If $A$ is a $\Gamma$-graded algebraic Azumaya $R$-algebra, the homotopy groups of the space $\Aut_{\Alg_R}(A)$ satisfy
\[
\pi_t(\Aut_{\Alg_R}(A), \mathrm{id}) \cong 
\begin{cases}
  \Aut_{\Alg_{\pi_\star R}}(\pi_\star A)&\text{if }t = 0,\\
  \pi_t A/ \pi_t R&\text{if }t > 0.
\end{cases}
\]
\item If $A$ is a $\Gamma$-graded algebraic Azumaya $R$-algebra, then for $t > 0$ the sequence
\[
0 \to \pi_t \GL_1(R) \to \pi_t \GL_1(A) \to \pi_t \Aut_{\Alg_R}(A) \to 0
\]
is exact, and there is an exact sequence of potentially nonabelian groups 
\[
1 \to \pi_0 \GL_1(R) \to 
\pi_0 \GL_1(A) \to 
\pi_0 \Aut_{\Alg_R}(A) \to \pi_0 \Pic(R).
\]
The image in $\pi_0 \Pic(R)$ of the last map is the group of outer automorphisms of $\pi_\star A$ as a $\pi_\star R$-algebra.
  \end{enumerate}
\end{theorem}

\begin{proof}
The Goerss--Hopkins obstruction groups $\Der^s_{\Alg_{\pi_\star R}}(\pi_\star A, M)$ appearing in Theorem~\ref{thm:obstruction-calcs} vanish identically for $s > 0$ by Proposition~\ref{prop:hochschild-vanishing}. In particular, the obstructions to existence and uniqueness vanish, so every Azumaya $\pi_\star R$-algebra lifts to an Azumaya $R$-algebra. Moreover, the obstructions to existence and uniqueness for lifting maps also vanish, and so every map of Azumaya $\pi_\star R$-algebras lifts uniquely to a map of $R$-algebras.

We then apply the long exact sequence on homotopy groups to Corollary~\ref{cor:fibersequence}. The previous theorem implies that this decomposes into short exact sequences on $\pi_t$ for $t > 1$ and the stated results on $\pi_1$ and $\pi_0$ once due caution is exercised regarding basepoints.
\end{proof}

\begin{corollary}
The functor $\pi_\star$ restricts to an equivalence from the homotopy category of algebraic $\Gamma$-graded Azumaya $R$-algebras to the category of Azumaya $\pi_\star R$-algebras.
\end{corollary}

\begin{remark}
There are two very common sources of non-algebraic Azumaya $R$-algebras. First, any compact generator $M$ of $\Mod_R$ produces an Azumaya $R$-algebra $\End_R(M)$ regardless of whether $\pi_\star M$ is projective or not (for example, the derived endomorphism ring of $\ZZ \oplus \ZZ/p$ is a non-algebraic derived Azumaya algebra over $\ZZ$). Second, the property of being algebraic also depends on the grading. If $P$ is an element in $\Pic(R)$ which is not a suspension of $R$, then $\End_R(R \oplus P)$ is likely to be exotic for $\ZZ$-grading but is definitely not exotic for $\Pic$-grading.
\end{remark}

\section{Presentable symmetric monoidal $\i$-categories}
\label{sec:commutative-algebras}

From this section forward, we will switch to an $\i$-categorical point of view on categories of Azumaya algebras and their module categories so that we can make use of the results of \cite{lurie-higheralgebra} and \cite{GH}.
Finding strict model-categorical versions of many of these constructions we will use seems extremely difficult. For example, it is hard to find point-set constructions that simultaneously give a construction of $GL_n(R)$ as a group, $M_n(R)$ as an $R$-algebra, an action of $GL_n(R)$ on $M_n(R)$ by conjugation, and a diagonal embedding $GL_1(R) \to GL_n(R)$ which acts trivially. If we also want these to be homotopically sensible then it becomes harder still.

Making this switch implicitly requires a translation process, which we will briefly sketch. Given a  commutative symmetric ring spectrum $R$, its image $\bar R$ in the $\i$-category $\S$ of spectra is a commutative algebra object in the sense of \cite[2.1.3.1]{lurie-higheralgebra}.
\begin{itemize}
\item \cite[4.1.3.10]{lurie-higheralgebra} Associated to $\Mod_R$ there is a stable presentable symmetric monoidal $\i$-category $\nrv^\otimes(\Mod_R^\circ)$, the operadic nerve of the category $(\Mod_R^\circ) \subset \Mod_R$ of cofibrant-fibrant $R$-modules.
\item \cite[4.3.3.17]{lurie-higheralgebra} This $\i$-category is equivalent to the $\i$-category of modules over the associated commutative algebra object $\bar R$ in $\S$.
\item \cite[4.1.4.4]{lurie-higheralgebra} The model category of associative algebra objects $\Alg_R$ has $\i$-category equivalent to the $\i$-category of associative algebra objects of $\nrv^\otimes(\Mod_R^\circ)$ in the sense of \cite[4.1.1.6]{lurie-higheralgebra}.
\item \cite[4.3.3.17]{lurie-higheralgebra} For such $R$-algebras, the model categories of left $A$-modules, right $A$-modules, or $A$-$B$ bimodules in $\Mod_R$ have associated $\i$-categories equivalent to the left modules, right modules, or bimodules over the corresponding associative algebra objects in $\nrv^\otimes(\Mod_R^\circ)$.
\end{itemize}

\begin{definition}
Let $\Ring:=\CAlg(\S)$ denote the $\i$-category of $\EE_\i$-ring spectra, or, equivalently, commutative algebra objects in $\S$.
\end{definition}
\subsection{Closed symmetric monoidal $\i$-categories}

\begin{definition}[\protect{\cite[4.1.1.7]{lurie-higheralgebra}}]
A monoidal $\i$-category $\C^\otimes$ is {\em closed} if, for each object $A$ of $\C$, the functors $A\otimes(-)\co \C\to\C$ and $(-)\otimes A\co \C\to\C$ admit right adjoints.
A symmetric monoidal $\i$-category $\C^\otimes$ is {\em closed} if the underlying monoidal $\i$-category is closed.
\end{definition}

Recall \cite[4.8]{lurie-higheralgebra} that the $\i$-category of
$\PrL$ of presentable $\i$-categories and colimit-preserving functors
\cite[5.5.3.1]{lurie-htt} admits a symmetric monoidal structure with unit the $\i$-category $\mathcal{S}$ of spaces.
We refer to (commutative) algebra objects in this $\i$-category as {\it presentable (symmetric) monoidal $\i$-categories}.

\begin{proposition}[\protect{\cite[4.2.1.33]{lurie-higheralgebra}}]
A presentable monoidal $\i$-category is closed.
\end{proposition}

\begin{proof}
Let $\C^\otimes$ be a presentable monoidal $\i$-category.
Then, by definition, the underlying $\i$-category $\C$ is presentable, and for each object $A$ of $\C$ the functors $A\otimes(-)$ and $(-)\otimes A$ commute with colimits.
It follows from the adjoint functor theorem \cite[5.5.2.2]{lurie-htt} that both of these functors admit right adjoints.
\end{proof}

Note that this implies that (the underlying $\i$-category of) a presentable symmetric monoidal $\i$-category $\C^\otimes$ is canonically enriched, tensored and cotensored over itself.
If $\C^\otimes$ is stable, then $\C$ is enriched, tensored and cotensored over $\S$, the $\i$-category of spectra.
We will not normally notationally distinguish between the internal mapping object and the mapping spectrum, which should always be clear from the context.

\begin{proposition}
A symmetric monoidal $\i$-category $\R$ is stable and presentable (as
a symmetric monoidal $\i$-category) if and only if the underlying
$\i$-category is stable and presentable and (any choice of) the tensor
bifunctor $\R\times\R\to\R$ preserves colimits in each variable.
In particular, a closed symmetric monoidal $\i$-category $\R$ is stable and presentable if and only if the underlying $\i$-category is stable and presentable.
\end{proposition}

There is also the following multiplicative version of Morita theory.

\begin{proposition}[{\cite[7.1.2.7]{lurie-higheralgebra}, \cite[3.1]{antieau-gepner-brauergroups}}]
The functor
\[
\Mod\co\CAlg(\S)\too\CAlg(\PrL),
\]
sending $R$ to the (symmetric monoidal, presentable, stable) $\i$-category of $R$-modules, is a fully faithful embedding.
\end{proposition}

\subsection{Structured fibrations}

We will write $\Cat^\land_\i$ for the very large $\i$-category of large $\i$-categories.

\begin{definition}
Given a (possibly large) $\i$-category $\C$ and a functor $\C\to{\Cat^\land_\i}$, we will say that a coCartesian fibration $X\to S$ admits a $\C$-structure if its classifying functor $X\to{\Cat^\land_\i}$ factors through $\C\to{\Cat^\land_\i}$.
\end{definition}

We have a coCartesian fibration $\Mod \to \Ring$ \cite[4.5.3.6]{lurie-higheralgebra}
whose fiber over the $\EE_\i$-ring spectrum $R$ is the (large) $\i$-category $\Mod_R$ of $R$-modules.

\begin{proposition}[\protect{\cite[4.5.3.1, 4.5.3.2]{lurie-higheralgebra}}]
The coCartesian fibration $\Mod\to\Ring$ admits a canonical symmetric
monoidal structure: there is a coCartesian family of $\i$-operads
\[
\Mod^\otimes \to \Ring \times \Comm^\otimes
\]
classifying a functor $R \mapsto \Mod_R\co \Ring \to
\CAlg(\PrL_{\mathrm{st}})$ from $\EE_\i$-ring spectra to presentable stable symmetric
monoidal $\i$-categories.
\end{proposition}

We next consider algebra objects. By applying
\cite[4.8.3.13]{lurie-higheralgebra}, we similarly find that we have a
coCartesian fibration $\Alg\to\Ring$ whose fiber over the ring $R$ is the (large) $\i$-category $\Alg_R$ of $R$-algebras.

\begin{proposition}
The coCartesian fibration $\Alg\to\Ring$ admits a canonical symmetric monoidal structure such that the forgetful functor from algebras to modules induces a morphism of symmetric monoidal coCartesian fibrations
\[
\xymatrix{\Alg\ar[rr]\ar[rd] & & \Mod\ar[ld]\\
& \Ring &}
\]
over $\Ring$.
\end{proposition}

\begin{proof}
As in \cite[5.3.1.20]{lurie-higheralgebra}, the coCartesian family of
$\i$-operads $\Mod^\otimes \to \Ring \times \Comm^\otimes$ classifies
a functor $\Ring \to (\Op_\i)^{/\Comm^\otimes}$, taking $R$ to the
coCartesian fibration $\Mod_R^\otimes \to \Comm^\otimes$. Applying
\cite[3.4.2.1]{lurie-higheralgebra}, we obtain a functor $\Alg\co
\Ring \to (\Op_\i)^{/\Comm^\otimes}$, taking $R$ to a coCartesian
fibration $\Alg_R^\otimes \to \Comm^\otimes$ with a forgetful map
\[
\xymatrix{\Alg^\otimes_R\ar[rr]\ar[rd] & & \Mod^\otimes_R\ar[ld]\\
& \Comm^\otimes &}
\]
that preserves coCartesian arrows
\cite[3.2.4.3]{lurie-higheralgebra}. Converting this back, we obtain a
diagram
\[
\xymatrix{\Alg^\otimes\ar[rr]\ar[rd] & & \Mod^\otimes\ar[ld]\\
& \Ring \times \Comm^\otimes &}
\]
of coCartesian $\Ring$-families of symmetric monoidal $\i$-operads,
lifting the underlying map $\Alg \to \Mod$ to one compatible with the symmetric monoidal structure.
\end{proof}

Restricting the coCartesian fibration $\Mod\to\Ring$ to the 
subcategory of coCartesian arrows between compact modules, we obtain a
left fibration
\[
\Mod^\omega\too\Ring
\] whose fiber over $R$ is the $\i$-groupoid $\Mod^\omega_R$ of compact (or perfect \cite[7.2.5.2]{lurie-higheralgebra}) $R$-modules and equivalences thereof. More specifically, an arrow $(R,M)\to (R',M')$ of $\Mod^\omega$ is an arrow $(R,M)\to (R',M')$ of $\Mod$ such that $M$ is compact (as an $R$-module) and the map $M\otimes_R R'\to M'$ is an equivalence.
Note that $\Mod^\omega_R$ should not be confused with the full subcategory of $\Mod_R$ spanned by the compact objects; rather, it is the full subgroupoid of $\Mod_R$ spanned by the compact objects.

Lastly, let $\Alg^\prop\to\Ring$ denote the left fibration whose source $\i$-category is the subcategory $\Alg^\prop$ of proper algebras defined by the pullback
\[
\xymatrix{\Alg^\prop\ar[r]\ar[d] & \Alg\ar[d]\\
\Mod^\omega\ar[r] & \Mod.}
\]
This time, however, $\Alg^\prop_R$ is not the full subgroupoid of $\Alg_R$ on the compact $R$-algebras, but rather the full subgroupoid of $\Alg_R$ consisting of the $R$-algebras $A$ whose underlying $R$-module is compact.

\begin{proposition}
The morphism of symmetric monoidal coCartesian fibrations $\Alg\too\Mod$ over $\Ring$ restricts to a morphism of symmetric monoidal left fibrations $\Alg^\prop\to\Mod^\omega$ over $\Ring$.
\end{proposition}

\begin{proof}
Tensors of compact modules are compact \cite[5.3.1.17]{lurie-higheralgebra}.
\end{proof}

\subsection{Functoriality of endomorphisms}

In order to construct the endomorphism algebra as a functor, we need to extend the results of \cite[4.7.2]{lurie-higheralgebra}. In this, Lurie considers the category of tuples $(A,M,\phi\co A \otimes M \to M)$, which has a forgetful functor $p$ given by $p(A,M,\phi) = M$. He extends it in such a way as to give this functor $p$ monoidal fibers; this gives the terminal object $\End(M)$ in the fiber over $M$ a canonical monoid structure. For the reader's convenience, we will first review some details of Lurie's construction.

Let $\L\M^\otimes$ denote the $\i$-operad parametrizing pairs of an algebra and a left module \cite[4.2.1.7]{lurie-higheralgebra}. A coCartesian fibration $\O^\otimes \to \L\M^\otimes$ of $\i$-operads determines a monoidal $\i$-category $\C$ and an $\i$-category $\M$ such that $\M$ is {\em left-tensored} over $\C$ \cite[4.2.1.19]{lurie-higheralgebra}: in particular, there exist objects $A \otimes M$ for $A \in \C$ and $M \in \M$.  Associated to this there is a category $\LMod(\M)$ of left module objects in $\M$ \cite[4.2.1.13]{lurie-higheralgebra}: such an object is determined by an algebra $A \in \C$ and a left $A$-module $M \in \M$. There is a forgetful map $\LMod(\M) \to \M$ which is a categorical fibration.

\begin{proposition}
Let $\Act(\M)$ be the fiber product $\LMod(\M) \times_{\M} \M^{\simeq}$. The natural map $\Act(\M) \to \M^\simeq$ is a coCartesian fibration.
\end{proposition}

\begin{proof}
The map $\Act(\M) \to \M$ is a categorical fibration to a Kan complex, and so by \cite[2.4.1.5, 2.4.6.5]{lurie-htt} it is a coCartesian fibration.
\end{proof}

\begin{definition}[{\cite[4.2.1.28]{lurie-higheralgebra}}]
Suppose that $\M$ is left-tensored over the monoidal $\i$-category $\C$. A {\em morphism object} for $M$ and $N$ is an object $F_\M(M,N)$ of $\C$ equipped with a map $F_\M(M,N) \otimes M \to N$ such that the resulting natural homotopy class of map
\[
\Map_\C(C,F_\M(M,N)) \to \Map_\M(C \otimes M,N)
\]
is a homotopy equivalence for all $C \in \C$. If morphism objects exist for all $M$ and $N$, we say that the left-tensor structure gives $\M$ a {\em $\C$-enrichment}.
\end{definition}

\begin{proposition}[{\cite[4.7.2.40]{lurie-higheralgebra}}]
Suppose that $\M$ is left-tensored over $\C$, giving it a $\C$-enrichment. For any $M \in \M$, the fiber $\LMod(\M) \times_\M \{M\}$ has a final object $\End(M)$ whose image under the composite $\LMod(\M) \to \Alg_\C \to \C$ is $F_\M(M,M)$.
\end{proposition}

\begin{corollary}
\label{cor:endfunc}
Under these assumptions, there exists a functor $\End\co \M^\simeq \to \Alg_\C$ sending $M$ to $\End(M)$.
\end{corollary}

\begin{proof}
By \cite[2.4.4.9]{lurie-htt}, the full subcategory of $\Act(\M)$ spanned by the final objects determines a trivial Kan fibration $\End(\M) \to \M^\simeq$. Choosing a section of this map, we obtain a composite functor
\[
\M^\simeq \to \LMod(\M) \to \Alg_\C
\]
with the desired properties.
\end{proof}

\begin{remark}
It should be sufficient to assume that $\C$ is monoidal and $\M$ merely $\C$-enriched, rather than including the stronger assumption that $\M$ is left-tensored over $\C$. However, we require this assumption in order to make use of the results from \cite[4.7.2]{lurie-higheralgebra}.
\end{remark}

\section{Picard and Brauer spectra}
\label{sec:picard-brauer}

\subsection{Picard spectra}
\label{sec:pic-spectra}

If $\C$ is a small $\i$-category, we write $\pi_0\C$ for the set of equivalence classes of objects of $\C$.
Note that, by definition, $\pi_0\C$ is an invariant of the underlying $\i$-groupoid $\C^\simeq$ of $\C$ (the $\i$-groupoid obtained by discarding the noninvertible arrows).

\begin{definition}
A symmetric monoidal $\i$-category $\C$ is grouplike if the monoid $\pi_0\C$ is a group.
\end{definition}

A symmetric monoidal $\i$-category $\C$ has a unique maximal grouplike symmetric monoidal subgroupoid $\C^\times$, the subcategory $\C^\times\subset\C$ consisting of the invertible objects and equivalences thereof.
That this is actually a symmetric monoidal subcategory in the $\i$-categorical sense follows from the fact that invertibility and equivalence are both detected upon passage to the symmetric monoidal homotopy category, and the grouplike condition is guaranteed by considering only the invertible objects.

Let $\PrL_{\mathrm{st}} \subset \PrL$ denote the $\i$-category of {\em stable} presentable $\i$-categories and colimit-preserving functors; by \cite[4.8.2.18]{lurie-higheralgebra} this is the category $\Mod_\S$ of left modules over the $\i$-category of spectra. We have the $\i$-category $\CAlg(\Mod_\S)$ of commutative ring objects in $\PrL_{\mathrm{st}}$; these are the same as commutative $\S$-algebras or presentable symmetric monoidal stable $\i$-categories.

\begin{definition}
Let $\R$ be a commutative $\S$-algebra.
The Picard $\i$-groupoid $\Pic(\R)$ of $\R$ is $\R^\times$, the maximal subgroupoid of the underlying $\i$-category of $\R$ spanned by the invertible objects.
\end{definition}

By \cite[8.9]{ando-blumberg-gepner-umkehr} $\Pic(\R)$ is equivalent to a small space, and by \cite[8.10]{ando-blumberg-gepner-umkehr} the functor $\Pic$ commutes with limits.

We have a symmetric monoidal cocartesian fibration
\[
\Mod(\Mod_\S)\too\CAlg(\Mod_\S)
\]
whose fiber over a commutative $\S$-algebra $\R$ is the symmetric monoidal $\i$-category $\Cat_\R$ of $\R$-linear $\i$-categories.
Writing 
\[
\Mod_{\S,\omega}\subset\Mod_\S
\]
for the symmetric monoidal subcategory consisting of the compactly-generated $\S$-modules and compact object preserving functors, this restricts to a symmetric monoidal cocartesian fibration
\[
\Mod(\Mod_{\S,\omega})\too\CAlg(\Mod_{\S,\omega})
\]
over the subcategory $\CAlg(\Mod_{\S,\omega})\subset\CAlg(\Mod_\S)$ of commutative algebra objects in $\Mod_{\S,\omega}\subset\Mod_\S$.
For a commutative algebra object $\R\in\CAlg(\Mod_{\S,\omega})$, a.k.a. a compactly generated commutative $\S$-algebra, we write $\Cat_{\R,\omega}^\simeq$ for the full subgroupoid of the fiber $\Cat_{\R,\omega}$ over $\R$, the symmetric monoidal $\i$-category of compactly generated $\R$-linear $\i$-categories in the sense of \cite[5.3.5]{lurie-htt}, and note that the map $\R \mapsto \Cat_{\R,\omega}^\simeq$ defines a left fibration $\Mod^\simeq(\Mod_{\S,\omega})\to\CAlg(\Mod_{\S,\omega})$.

\begin{proposition}[{cf. \cite[2.1.3]{mathew-stojanoska-picard}}]
Let $\R$ be a compactly-generated commutative $\S$-algebra.
Then any invertible object of $\R$ is compact.
\end{proposition}

\begin{proof}
Let $I$ be any invertible object of $\R$ and let $\{M_\alpha\}$ be a filtered system of objects of $\R$.
Then there are natural equivalences
\begin{align*}
\Map(I,\colim M_\alpha)
&\simeq\Map({\bf 1},\colim I^{-1}\otimes M_\alpha)\\
&\simeq\colim\Map({\bf 1},I^{-1}\otimes M_\alpha)\\
&\simeq\colim\Map(I,M_\alpha).
\end{align*}
The first equivalence follows because $\otimes$ commutes with
colimits. The second follows because the monoidal unit ${\bf 1}$
(the image of the sphere spectrum under the map $\S \to \R$) is compact by definition.
\end{proof}

Because $\Pic(\R)$ is closed under the symmetric monoidal product on
$\R$, it is a grouplike symmetric monoidal $\i$-groupoid, so by the recognition principle for infinite loop spaces we may
regard $\Pic(\R)$ as having an associated (connective) spectrum
$\pic(\R) = K(\Pic(\R))$ \cite[1.4]{clausen}.
Let $\Gamma_\R$ be the algebraic Picard groupoid of $\R$: the homotopy category of $\Pic(\R)$, which is the $1$-truncation of $\Pic(\R)$.  If $\R$ is unambiguous, we drop it and simply write $\Gamma$.  We will notationally distinguish between an object $\gamma \in \Gamma$ and the associated invertible object $R^\gamma \in \R$.

\begin{proposition}
The homotopy category of $\R$ is canonically enriched in the symmetric monoidal category of $\Gamma$-graded abelian groups.
\end{proposition}

\begin{proof}
Since $\R$ is stable, the set $\pi_0\Map(M,N)$ of homotopy classes of maps from $M$ to $N$ admits an abelian group structure which is natural in the variables $M$ and $N$ of $\R$, and composition is bilinear. The result then follows from Proposition~\ref{prop:graded-enrich}, defining $\pi_\star\Map(M,N)$ by the rule
\[
\pi_\gamma\Map(M,N):=\pi_0\Map(\Sigma^\gamma M,N).\qedhere
\]
\end{proof}

If $R$ is an $\EE_\i$-ring spectrum, then we will typically write $\Pic(R)$ in place of $\Pic(\Mod_R)$ and $\Gamma_R$ in place of $\Gamma_{\Mod_R}$.

\subsection{Brauer spectra}

\begin{definition}
Let $\R$ be a compactly generated $\S$-algebra. The Brauer $\i$-groupoid $\Br(\R)$ of $\R$ is the subgroupoid $\Pic(\Cat_{\R,\omega})$ of $\Cat_{\R,\omega}$ on the invertible $\R$-linear categories which admit a compact generator.
\end{definition}

\begin{remark}
If $\C$ is a presentable $\i$-category, then $\C\simeq\mathrm{Ind}_\kappa(\C^\kappa)$ is the $\kappa$-filtered colimit completion of the full subcategory $\C^\kappa\subset\C$ on the $\kappa$-compact objects for some sufficiently large cardinal $\kappa$.
If $\kappa$ can be taken to be countable, then $\C$ is said to be compactly generated, and if there exists a compact object $P\in\C$ such that $\C\simeq\Mod_{\End(P)}$ as $\S$-modules, then $\C$ is said to admit a compact generator.
Note that an $\S$-module $\C$ admits a compact generator $P$ if and only if the smallest thick subcategory of $\C^\omega$ containing $P$ is $\C^\omega$ itself, in which case $\C^\omega\simeq\Mod_{\End(P)}^\omega$. 
Also observe that there is a distinction between these objects---$\R$-linear $\i$-categories with a compact generator---and the compact objects in $\Cat_\R$.
\end{remark}

Because $\Br(\R)$ is closed under the symmetric monoidal product on $\Cat_\R$, it is a grouplike symmetric monoidal $\i$-groupoid, so we may associate to it a connective spectrum $\br(\R)$.

\begin{proposition}
Let $\R$ be a compactly-generated stable symmetric monoidal $\i$-category. Then there is a canonical equivalence $\Pic(\R)\to\Omega\Br(\R)$.
\end{proposition}

\begin{proof}
We must show that $\Omega\Pic(\Cat_{\R,\omega})\simeq\Pic(\R)$.
To see this, first note that if $X$ is a pointed $\i$-groupoid, then $\Omega X\simeq\Aut_X(*)$ is the space of automorphisms of the distinguished object $*$ of $X$.
Hence $\Omega\Pic(\Mod_\R^\omega)\simeq\Aut_{\R}(\R)\simeq\Pic(\R)$, where the last equivalence follows from the fact that invertible $\R$-module endomorphisms of $\R$ correspond to invertible objects of $\R$ under the equivalence $\End_{\R}(\R)\simeq\R$ \cite[4.8.4]{lurie-higheralgebra}.
\end{proof}

If $R$ is an $\EE_\i$-ring spectrum, we will typically write $\Br_R$ for the $\i$-groupoid $\Br(\Mod_R)$.

\subsection{Azumaya algebras}

\begin{definition}[{\cite{baker-richter-szymik}, \cite[3.1.3]{antieau-gepner-brauergroups}, \cite{toen-azumaya}}]
Let $R$ be an $\EE_\i$-ring spectrum.
An Azumaya $R$-algebra is an $R$-algebra $A$ such that 
\begin{itemize}
\item
the underlying $R$-module of $A$ is a compact generator of $\Mod_R$ in the sense of \cite[5.5.8.23]{lurie-htt}, and
\item
the ``left-and-right'' multiplication map $A\otimes_R A^\op\to\End_R(A)$ is an $R$-algebra equivalence.
\end{itemize}
\end{definition}

\begin{remark}
In \cite[3.15]{antieau-gepner-brauergroups} it is shown that such $A$ are characterized by the fact that $\Mod_A$ is invertible in the cateory $\Cat_{R,\omega}$ of compactly generated $\Mod_R$-linear $\i$-categories.
\end{remark}

\begin{proposition}
If $A$ is an Azumaya $R$-algebra and $R\to R'$ is a ring map, then $A\otimes_R R'$ is an Azumaya $R'$-algebra.
\end{proposition}

We write $\Az$ for the full subcategory of $\Alg$ determined by pairs $(R,A)$ such that $A$ is an Azumaya $R$-algebra.
Because Azumaya algebras are stable under base-change, we have a morphism of left fibrations
\[
\xymatrix{\Az\ar[rr]\ar[rd] & & \Alg\ar[ld]\\
& \Ring &}
\]
over $\Ring$.

\begin{proposition}
Let $A$ be an Azumaya $R$-algebra.
Then the category of right $A$-modules $\Mod_A$ is an invertible $\Mod_R$-module with inverse $\Mod_{A^\op}$.
\end{proposition}

\begin{proof}
We must show that $\Mod_{A}\otimes\Mod_{A^\op}\simeq\Mod_R$, where the tensor is taken in the category of left $\Mod_R$-linear categories.
Since $\Mod$ is symmetric monoidal \cite[4.8.5.16]{lurie-higheralgebra}, we have an equivalence $\Mod_{A}\otimes_{\Mod_R}\Mod_{A^\op}\simeq\Mod_{A\otimes_R A^\op}$, and as ``left-and-right'' multiplication $A\otimes_R A^\op\to\End_R(A)$ is an equivalence of $R$-algebras we see that $\Mod_{A}\otimes\Mod_{A^\op}\simeq\Mod_{\End_R(A)}$.
Finally, because $A$ is a compact generator of $\Mod_R$, Morita theory gives an equivalence $\Mod_R\simeq\Mod_{\End_R(A)}$ \cite[8.1.2.1]{lurie-higheralgebra}, and the result follows.
\end{proof}

\begin{remark}
We can instead show that the functor $\Mod$ itself is symmetric monoidal using the results of \cite{blumberg-gepner-tabuada-cyclotomic}. There it is shown that that the category of stable $\i$-categories is the symmetric monoidal localization of the category of spectral $\i$-categories, obtained by inverting the Morita equivalences. In particular, regarding ring spectra $A$ and $B$ as one-object spectral $\i$-categories, it follows that $\Mod_A\otimes \Mod_B\simeq \Mod_{A\otimes B}$. The relative tensors are computed as the geometric realization of two-sided bar constructions $B(A,R,B)$ and $B(\Mod_A,\Mod_R,\Mod_B)$ \cite[4.4.2.8]{lurie-higheralgebra}; the localization functor preserves geometric realization due to being a left adjoint.
\end{remark}

\begin{proposition}
The map of $\i$-groupoids $\Az_R\to\Br_R$ is essentially surjective.
Moreover, if $A$ and $B$ are Azumaya algebras such that the images of $A$ and $B$ become equal in $\pi_0\Br_R$, then $A$ and $B$ are Morita equivalent.
\end{proposition}

\begin{proof}
Let $\R=\Mod_R$ and let $\I$ be an invertible object of $\Cat_{\R,\omega}$.
Then $\I$ has a compact generator, so $\I\simeq\Mod_A$ for some $R$-algebra $A$ (\cite{schwede-shipley}, \cite[7.1.2.1]{lurie-higheralgebra}), and invertibility implies that $A$ is a compact generator of $\Mod_R$.
It follows that $\End_R(A)$ is Morita equivalent to $R$, and thus that the $R$-algebra map $A\otimes_R A^\op\to\End_R(A)$ is an equivalence.
\end{proof}

We remark that we can identify the homotopy types of the fibers of the various left fibrations over $\Ring$.

\begin{proposition}
Let $R$ be an $\EE_\i$-ring spectrum.
Then
\[
\Mod^{\omega}_R\simeq\coprod_{[M]\in\pi_0\Mod^{\omega}_R} B\Aut_R(M)
\]
and
\[
\Az_R\simeq\coprod_{[A]\in\pi_0\Az_R} B\Aut_{\Alg_R}(A).
\]
\end{proposition}

\subsection{The conjugation action on endomorphisms}

Let $\R$ be a symmetric monoidal presentable stable $\i$-category with unit ${\bf 1}$, which is therefore enriched over itself (see \cite[4.2.1.33]{lurie-higheralgebra} or \cite[7.4.10]{GH})), and let $M$ be an object of $\R$.
In this section we analyze the fiber of the map
\[
\Aut_\R(M)\too\Aut_{\Alg_\R}(\End_\R(M)),
\]
which roughly sends an automorphism $\alpha$ of $M$ to the conjugation automorphism $\alpha^{-1}\circ(-)\circ\alpha$ of the endomorphism algebra $\End(M)$. This map arises from the map $\End_{\R}\co \R^\simeq\to\Alg_\R$ of Corollary \ref{cor:endfunc}.

\begin{proposition}
\label{prop:bglquotient}
Let $R$ be an $\EE_\i$ ring spectrum, $A$ an Azumaya $R$-algebra, and $\Mod_A^{\cg}$ denote the $\i$-category of compact generators of $\Mod_A$. Then there are canonical pullback diagrams of $\i$-categories
\[
\xymatrix{
\Pic(R) \ar[r] \ar[d] &
\Mod_A^{\cg} \ar[r] \ar[d] &
\{\Mod_A\} \ar[d] \\
\{A\} \ar[r] &
\Az_R \ar[r] &
\Br_R.
}
\]
More generally, the fiber of $\Mod^{\cg}_A \to \Alg_R$
over an $R$-algebra $B$ is either empty or a principal torsor for $\Pic(R)$.
\end{proposition}

\begin{remark}
  Note that $\Mod_A^{\cg}$ is not to be confused with the larger subcategory $\Mod_A^\omega$ of compact objects.
\end{remark}
\begin{proof}
The pullback of the right-hand square is the $\i$-category of $R$-algebras equipped with a Morita equivalence to $\Mod_A$. In \cite[4.8.4]{lurie-higheralgebra} it is shown that the category of functors $\Mod_A \to \Mod_B$ is equivalent to a category of bimodules, and so this pullback category of Morita equivalences is equivalent to the category of compact generators of $\Mod_A$ via the map $M \mapsto \End_A(M)$.

The pullback of the left-hand square is the $\i$-category of $A$-bimodules inducing $\Mod_R$-linear Morita self-equivalences of $\Mod_A$. These are, in particular, invertible modules over $A \otimes_R A^{op} \simeq \End_R(A)$, and Morita theory implies that the map $I \mapsto I \otimes_R A$ makes this equivalent to $\Pic(R)$.
\end{proof}

Taking preimages of the unit component, we obtain the following.
\begin{corollary}
\label{cor:fibersequence}
For an $\EE_\i$-ring $R$, there is a fiber sequence
\[
\coprod_{[M] \in \pi_0 \Mod_R^{\cg}} B\Aut_R(M) \to 
\coprod_{[A] \in \pi_0 (\Az_R)^{triv}} B\Aut_{\Alg_{R}}(A) \to
B\Pic (R),
\]
where the middle coproduct is over Azumaya $R$-algebras Morita equivalent to $R$.
\end{corollary}

In particular, this implies that the map $\Aut_R(M) \to \Aut_{\Alg_{R}}(\End_R(M))$ factors through a quotient by $\GL_1(R)$.

\begin{corollary}[cf. \ref{cor:rz}]
\label{cor:pgl-sequence}
For any Azumaya $R$-algebra $A$, there is a long exact sequence of groups
\begin{align*}
\cdots &\to (\pi_n R)^\times \to (\pi_n A)^\times \to \pi_n (\Aut_{\Alg_R}(A)) \to \cdots\\
&\to (\pi_0 R)^\times \to (\pi_0 A)^\times \to \pi_0 (\Aut_{\Alg_R}(A)) \to \pi_0 \Pic(R).
\end{align*}
Moreover, the group $\pi_0 \Pic(R)$ acts on the set of isomorphism classes of compact generators of $\Mod_A$. The quotient is the set of isomorphism classes of Azumaya algebras $A$ Morita equivalent to $R$, and the stabilizer of $A$ is the image of the group of outer automorphisms of $\pi_\star A$ as a $\pi_\star R$-algebra.
\end{corollary}

This long exact sequence generalizes the short exact sequences of Theorem~\ref{thm:obstruction-calcs} for $\Gamma$-graded algebraic Azumaya $R$-algebras.

\section{Galois cohomology}
\label{sec:galois-cohomology}

\subsection{Galois extensions and descent}

In this section we will review definitions of Galois extensions of ring spectra, due to Rognes \cite{rognes-galois}. Let $R$ be an $\EE_\i$-ring spectrum and let $G$ be an $R$-dualizable $\i$-group: $G\simeq\Omega BG$ is the space of automorphisms of an object in some pointed, connected $\i$-groupoid $BG$, and the associated group ring $R[G]:=R\otimes_{\SS}\Sigma^\infty_+G$ is dualizable as an $R$-module.

\begin{definition}
A Galois extension of $R$ by $G$ is a functor $f\co BG \to \Ring_{R/}$, sending the basepoint to a commutative $R$-algebra $S$ with $G$-action, such that
\begin{itemize}
\item
the unit map $R\to S^{hG} = \lim f$ is an equivalence, and
\item
the map $S\otimes_R S\to S\otimes_R D_R R[G]\simeq D_{S} S[G]$, induced by the action $R[G]\otimes_R S\to S$, is an equivalence.
\end{itemize}
A $G$-Galois extension $R\to S$ is {\em faithful} if $S$ is a faithful $R$-module.
\end{definition}

We will usually just write $f\co R\to S$ for the Galois extension without explicitly mentioning the $G$-action. All of the Galois extensions that we consider in this paper will be assumed to be faithful.

We have the following important result.
\begin{proposition}[{\cite[6.2.1]{rognes-galois}}]
Let $R \to S$ be a $G$-Galois extension. Then the underlying $R$-module of $S$ is dualizable.
\end{proposition}

In other words, $S$ is a {\em proper} $R$-algebra in the sense of \cite[4.6.4.2]{lurie-higheralgebra}. Using this, Mathew has deduced several important consequences.

\begin{proposition}[{\cite[3.36]{mathew-galois}}]
Let $R \to S$ be a faithful $G$-Galois extension with $G$ a finite group. Then $R \to S$ {\em admits descent} in the sense of \cite[3.17]{mathew-galois}.
\end{proposition}

\begin{proposition}
\label{prop:module-descent}
Let $R \to S$ be a faithful $G$-Galois extension with $G$ a finite group, and $M$ an $R$-module. Then several properties of $S$-modules descend:
\begin{itemize}
\item A map $M \to N$ of $R$-modules is an equivalence if and only if $S \otimes_R M \to S \otimes_R N$ is an equivalence.
\item $M$ is a faithful $R$-module if and only if $S \otimes_R M$ is a faithful $S$-module.
\item $M$ is a perfect $R$-module if and only if $S \otimes_R M$ is a perfect $S$-module.
\item $M$ is an invertible $R$-module if and only if $S \otimes_R M$ is an invertible $S$-module.
\end{itemize}
\end{proposition}

\begin{proof}
The first statement is equivalent to the statement that $N/M$ is trivial if and only if $S \otimes_R N/M$ is, which is the definition of faithfulness. The second statement follows from the tensor associativity equivalence $N \otimes_S (S \otimes_R M) \simeq N \otimes_R M$. The third statement is {\cite[3.27]{mathew-galois}} and the fourth is  {\cite[3.29]{mathew-galois}}.
\end{proof}

Associated to a commutative $R$-algebra $S$, there is the associated Amitsur complex, a cosimplicial commutative $R$-algebra:
\[
S^{\otimes_R} := \left\{ S\rrarrow S\otimes_R S\rrrarrow S\otimes_R S\otimes_R S\rrrrarrow\cdots\right\}
\]
In degree $n$ this is the $(n+1)$-fold tensor power of $S$ over $R$. More explicitly, the Amitsur complex is the left Kan extension of the map $\{[0]\} \to \Ring_{R/}$ classifying $S$ along the inclusion $\{[0]\} \into \Delta$.

Composing with the functor $\Mod\co \CAlg \to \Cat^\land_\i$, we obtain a cosimplicial object
\[
\Mod_{S^{\otimes_R}} :=\left\{\Mod_S\rrarrow\Mod_{S\otimes_R S}\rrrarrow\Mod_{S\otimes_R S\otimes_R S}\rrrrarrow\cdots\right\}
\]
in $\Mod_R$-linear $\i$-categories; we also refer to this as the Amitsur complex.

\begin{proposition}[{cf. \cite[6.15, 6.18]{lurie-dag7}, \cite[3.21]{mathew-galois}}]
\label{prop:amitsur-limit}
Suppose $S$ is a proper commutative $R$-algebra and $A$ is an $R$-algebra. Then the natural map
\[
\theta\co \Mod_A \to \lim \Mod_{(S^{\otimes_R}) \otimes_R A}
\]
has fully faithful left and right adjoints. If $S$ is faithful as an
$R$-module, then $\theta$ is an equivalence.
\end{proposition}

\begin{proof}
We will prove this result by verifying the two criteria of \cite[4.7.6.3]{lurie-higheralgebra} (a consequence of the $\i$-categorical Barr--Beck theorem) for both this cosimplicial diagram of categories and the corresponding diagram of opposite categories.

The first criterion asks that colimits of simplicial objects exist in $\Mod_A$ and that the extension-of-scalars functor $S \otimes_R (-)\co \Mod_A \to \Mod_{S \otimes_R A}$ preserve them. However, both categories are cocomplete and the given functor is left adjoint to the forgetful functor, hence preserves all colimits. The same condition on the opposite category asks that $S \otimes_R (-)$ preserve totalizations of certain cosimplicial objects, but since $S$ is $R$-dualizable there is a natural equivalence
\[
S \otimes_R M \simeq F_R(D_RS,M).
\]
This equivalent functor has a left adjoint, given by $N \mapsto D_R S \otimes_S M$, and so preserves all limits.

The second criterion is a ``Beck--Chevalley'' condition, as follows. For any $\alpha\co [m] \to [n]$ in $\Delta$, consider the induced diagram of $\i$-categories
\[
\xymatrix{
\Mod_{S^{\otimes_R m} \otimes_R A} \ar[r]^-{d^0} \ar[d] &
\Mod_{S^{\otimes_R (1+m)} \otimes_R A} \ar[d] \\
\Mod_{S^{\otimes_R n} \otimes_R A} \ar[r]^-{d^0} &
\Mod_{S^{\otimes_R (1+n)} \otimes_R A}.
}
\]
Then we ask that these diagrams are left adjointable and right adjointable \cite[4.7.5.13]{lurie-higheralgebra}: the horizontal arrows admit left and right adjoints, and that the resulting natural transformation between the composites is an equivalence. In our case, this diagram is generically of the following form:
\[
\xymatrix{
\Mod_B \ar[r] \ar[d] &
\Mod_{S \otimes_R B} \ar[d] \\
\Mod_{B'} \ar[r] &
\Mod_{S \otimes_R B'}
}
\]
Here the horizontal arrows are extension of scalars to $S$, while the vertical arrows are extensions of scalars induced by a map of $R$-algebras $B \to B'$. The natural transformation between composed left adjoints is the natural equivalence
\[
(D_RS \otimes_S M) \otimes_B B' \to D_RS \otimes_S (M \otimes_B B'),
\]
and the one between composed right adjoints is the natural equivalence
\[
N \otimes_{S \otimes_R B}(S \otimes_R B') \to N \otimes_B B',
\]
verifying the Beck--Chevalley condition and its opposite.

Therefore, the map from $\Mod_A$ to the limit category has fully faithful left and right adjoints. If $S$ is faithful, then the functor $\Mod_A \to \Mod_{S \otimes_R A}$ is conservative and \cite[4.7.6.3]{lurie-higheralgebra} additionally verifies that $\Mod_A$ is equivalent to the limit, making $\Mod_A$ monadic and comonadic over $\Mod_{S \otimes_R A}$.
\end{proof}

\begin{remark}
This construction has a stricter lift. If we lift $R$ and $S$ to strictly commutative ring objects in a model category and $G$ is an honest group acting on $S$, the operation of tensoring with the right $R$-module $S$ implements a left Quillen functor between the category of $R$-modules and the category of modules over the twisted group algebra $S\langle G\rangle \simeq \End_R(S)$.
\end{remark}

\subsection{Group actions}

Let $G\simeq\Omega BG$ be the space of automorphisms of a connected Kan complex $BG$ with basepoint $i\co \Delta^0 \to BG$. For an $\i$-category $\C$, the category of $G$-objects in $\C$ is the functor category $\C^{BG} = \Fun(BG,\C)$. Evaluation at the basepoint determines a functor $i^*\co \C^{BG} \to \C$.

If $\C$ is complete and cocomplete, the functor $i^*$ admits left and right adjoints $i_!$ and $i_*$ respectively, given by left and right Kan extension. These are naturally described by the colimit and limit of the constant diagram on $G$ with value $X$, or equivalently the tensor and cotensor of $X$ with $G$:
\begin{align*}
  i_! X &\simeq G \otimes X, &i_* X &\simeq X^G.
\end{align*}
\begin{proposition}
If $\C$ is complete and cocomplete, the forgetful functor $i^*$ makes $\C^{BG}$ monadic and comonadic over $\C$ in the sense of \cite[4.7.4.4]{lurie-higheralgebra}.

Suppose $G$ is equivalent to a finite discrete group and let $p\co \C \to \D$ be a functor between complete and cocomplete $\i$-categories which preserves finite products, with induced map $p_*\co \C^{BG} \to \D^{BG}$. If $\TT^\C$ and $\TT^\D$ are the induced comonads on $\C$ and $\D$, then the resulting natural transformation $p \circ \TT^\C \to \TT^\D \circ p$ between comonads is an equivalence.
\end{proposition}

\begin{proof}
For the first statement it suffices, by \cite[4.7.4.5]{lurie-higheralgebra} and its dual, to observe that $i^*$ is conservative and preserves all limits and colimits, being both a left and right adjoint.

For the second statement, the natural map is provided by the adjunction in the form of a composite
\[
p i^* i_* \cong i^* p_* i_* \xrightarrow{i^*(\eta)} i^* i_* i^* p_* i_* \cong i^* i_* p i^* i_* \xrightarrow{i^* i_* p(\epsilon)} i^* i_* p.
\]
For $X \in \C$, this takes the form of the natural limit transformation $p(X^G) \to p(X)^G$, which is an equivalence by assumption.
\end{proof}

\begin{corollary}
\label{cor:fixedpoint-constr}
If $G$ is a finite group, associated to a ring $S \in (\Ring_{R/})^{BG}$ there is a cosimplicial commutative $R$-algebra
\[
\TT^\bullet(S) = \left\{ i^* S\rrarrow \TT(i^*S)\rrrarrow \TT(\TT(i^*S)) \rrrrarrow\cdots\right\}
\]
induced by the comonad $\TT$, whose underlying cosimplicial $R$-module is the fixed-point construction
\[
S^{hG}=\left\{ S\rrarrow S^G \rrrarrow S^{G \times G} \rrrrarrow\cdots\right\}.
\]

\end{corollary}

\subsection{Descent}

\begin{proposition}
\label{prop:cosimp-equiv}
For a Galois extension $R \to S$, there is an equivalence of cosimplicial $R$-algebras between the Amitsur complex $S^\otimes_R$ and the fixed-point construction $\TT^\bullet(S)$ of Corollary~\ref{cor:fixedpoint-constr}.
\end{proposition}

\begin{proof}
The universal property of the left Kan extension implies that the identity map $S \simeq \TT^0(S)$ extends to a map of cosimplicial objects $S^{\otimes_R} \to S^{hG}$, unique up to contractible choice. It suffices to verify that this induces equivalences $S^{\otimes_R (n+1)} \to S^{G^n}$, which follows by induction from the case $n=1$.
\end{proof}

Now suppose that $R\to S$ is a faithful Galois extension of $R$ by the stably dualizable group $G$.
Write $f\co BG\to\Ring_{R/}$ for the functor classifying $S$ as a (naive) $G$-equivariant commutative $R$-algebra, so that $S\simeq f(*)$ and $R\simeq\lim f$, and write
\[
(\Mod_S)^{hG}:=\lim\Mod_f
\]
for the ``fixed-points'' of the $\i$-category $\Mod_f$, the $\i$-category of $G$-semilinear $S$-modules.
Lastly, we write $N^{hG}$ for the limit of a $G$-semilinear $S$-module $N$, and view it as an $R\simeq S^{hG}$-module.

\begin{theorem}
\label{thm:fixedpoint-categories}
Let $R\to S$ be a faithful $G$-Galois extension with $G$ finite, and $A \in \Alg_R$. Then the canonical map
\[
\Mod_A\too(\Mod_{S \otimes A})^{hG}
\]
is an equivalence of $\i$-categories.
\end{theorem}

\begin{proof}
Tensoring the equivalence of Proposition \ref{prop:cosimp-equiv} with $A$, we obtain maps of cosimplicial objects
\[
(S^{\otimes_R}) \otimes_R A \xrightarrow{\sim} \TT^\bullet(S) \otimes_R A \to \TT^\bullet(A).
\]
The natural map $\TT(X) \otimes_R Y \to \TT(X \otimes_R Y)$ is equivalent to the map $X^G \otimes_R Y \to (X \otimes_R Y)^G$ and is therefore an equivalence, because both sides are a $|G|$-fold coproduct of copies of $X \otimes_R Y$.

Since $S$ is faithful and dualizable as an $R$-module, Proposition~\ref{prop:amitsur-limit} shows that there is an equivalence
\[
\Mod_A \simeq \lim (\Mod_{S^{\otimes_R} \otimes_R A}).
\]
The equivalence of cosimplicial rings shows that this extends to an equivalence
\[
\Mod_A \simeq \lim \Mod_{\TT^\bullet A} \simeq (\Mod_{S \otimes_R A})^{hG}.\qedhere
\]
\end{proof}

\begin{corollary}
\label{cor:section-equivalence}
Let $R\to S$ be a faithful $G$-Galois extension with $G$ finite, associated to a functor $f\co BG\to\Ring_{R/}$, and consider the diagram
\[
\xymatrix{
& \Mod\ar[d]\\
BG\ar[r]^f \ar@{.>}[ur] & \Ring.
}
\]
Then the map
\[
\Mod_R\too\Fun_{/\Ring}(BG,\Mod),
\]
which sends the $R$-module $M$ to the $G$-Galois module $S \otimes_R M$, is an equivalence.
\end{corollary}

\begin{proof}
The $\i$-category of sections from $BG$ to the pullback of $\Mod \to \Ring$ is equivalent to the limit of the functor $\Mod_f\co BG\to{\Cat^\land_\i}$ it classifies \cite[3.3.3.2]{lurie-htt}, which in turn is equivalent to $\Mod_R$ by Theorem \ref{thm:fixedpoint-categories}.
\end{proof}

\begin{lemma}
For an $\i$-operad $\O$, the $\i$-category of $\O$-monoidal $\i$-categories has limits which are computed in $\Cat_\i$.
\end{lemma}

\begin{proof}
In \cite[2.4.2.6]{lurie-higheralgebra} it is shown that there is an equivalence between $\O$-monoidal $\i$-categories and $\O$-algebra objects in $\Cat_\i$, and so \cite[3.2.2.1]{lurie-higheralgebra} shows that limits of the underlying $\i$-categories lift uniquely to limits of $\O$-monoidal $\i$-categories. The same proof applies within the category of large $\i$-categories.
\end{proof}

\begin{corollary}
\label{cor:monoidal-limits}
Let $f\co I\to{\Cat^{\O}_\i}$ be a diagram of $\O$-monoidal $\i$-categories and $\O$-monoidal functors. Then the canonical map
\[
\Alg_{/\O}(\lim f) \to \lim (\Alg_{/\O} \circ f)
\]
is an equivalence.
\end{corollary}

\begin{proof}
The $\i$-category $\Alg_{/\O}(\C^\otimes)$ of $\O$-algebra objects in an $\O$-monoidal $\i$-category $p\co \C^\otimes \to \O^\otimes$ is the $\i$-category of functors $\Fun_{/\O^\otimes}(\O^\otimes, \C^\otimes)$, and $\Fun_{/\O^\otimes}(\O^\otimes,-)\co \Cat_\i^\O\to\Cat_\i$ evidently preserves limits in the target.
\end{proof}

\begin{proposition}
Let $R\to S$ be the $G$-Galois extension associated to a functor $f\co BG\to\Ring_{R/}$, and consider the diagram
\[
\xymatrix{
& \Alg\ar[d]\\
BG\ar[r]^f \ar@{.>}[ur] & \Ring.
}
\]
Then the map $\Alg_R\too\Fun_{/\Ring}(BG,\Alg)$, which sends the $R$-algebra $A$ to the $G$-equivariant $S$-algebra $S \otimes_R A$, is an equivalence.
\end{proposition}

\begin{proof}
This follows from the corresponding statement for modules, by noting that $f$ is comes from a diagram $BG\to\CAlg(\Cat_\i)$ of symmetric monoidal $\i$-categories and symmetric monoidal functors, together with Corollary~\ref{cor:monoidal-limits}.
\end{proof}

We now consider the diagram of $\i$-categories
\[
\xymatrix{
&\Pic \ar[dr] \ar[r] & \Mod^{\cg} \ar[d] \ar[r] & \Az \ar[dl] \ar[r] & \Br \ar[dll] \\
BG^\lhd \ar[rr] && \Ring,
}
\]
where the bottom map describes $R$ as the limit of the $G$-action on $S$. For each of the vertical maps we may take spaces of sections over the cone point or over $BG$, recovering fixed-point objects for the action of $G$ on $\Pic_S$, $\Mod_S^{\cg}$, $\Az_S$, and $\Br_S$ respectively.

\begin{theorem}
\label{thm:descentdiag}
Let $R\to S$ be the $G$-Galois extension with $G$ finite. There is a commutative diagram of symmetric monoidal $\i$-categories
\[
\xymatrix{
\Pic_R \ar[d] \ar[r] & 
\Mod^{\cg}_R \ar[d] \ar[r] & 
\Az_R \ar[d] \ar[r] & 
\Cat_{R,\omega} \ar[d] \\
(\Pic_S)^{hG} \ar[r] &
(\Mod^{\cg}_S)^{hG} \ar[r] &
(\Az_S)^{hG} \ar[r] &
(\Cat_{S,\omega})^{hG}
}
\]
where the left three vertical arrows are equivalences. On the full subcategory of $\Cat_{R,\omega}$ spanned by categories of the form $\Mod_A$ for $A$ an $R$-algebra, the right-hand map is fully faithful.
\end{theorem}

\begin{proof}
We already have equivalences $\Mod_R \simeq (\Mod_S)^{hG}$ and $\Alg_R \simeq (\Alg_S)^{hG}$, so for the left three arrows it suffices to identify the essential images of the $\i$-categories $\Pic_R$, $\Mod_R^{\cg}$, and $\Az_R$ of invertible modules, compact generators, and Azumaya algebras.

Proposition~\ref{prop:module-descent} implies that the property of  being invertible descends, as do subcategories of equivalences, so that the essential image of $\Pic_R$ is the subcategory $(\Pic_S)^{hG}$ of $(\Mod_S)^{hG}$. Similarly, Proposition~\ref{prop:module-descent} implies that the properties of being dualizable and faithful descend, and that for a dualizable $R$-algebra $A$ the map $A \otimes_R A^{op} \to \End_R(A) \simeq D_RA \otimes_R A$ is an equivalence if and only if the same is true for the $S$-algebra $S \otimes_R A$. Therefore, the essential image of $\Az_R$ is the subcategory $(\Az_S)^{hG}$ of $(\Alg_S)^{hG}$.

Compact generators are taken to compact generators, and so the second vertical arrow is defined. Further, if $M$ is an $R$-module whose image in $\Mod_S$ is a compact generator, then $M$ is compact and $\End_R(M)$ is an $R$-algebra whose image $\End_S(S \otimes_R M)$ is an Azumaya $S$-algebra, as already shown. Therefore $\End_R(M)$ is an Azumaya $R$-algebra, implying that $M$ is a generator.

It remains to verify full faithfulness of the right-hand map. We have a commutative diagram
\[
\xymatrix{
(\Mod^\simeq_{A \otimes_R B^\op}) \ar[r] \ar[d] &
(\Mod^\simeq_{(S \otimes_R A) \otimes_S (S \otimes_R B)^\op})^{hG} \ar[d] \\
\Map_{\Cat_{R}}(\Mod_A,\Mod_B) \ar[r] &
\Map_{\Cat_{S}}(\Mod_{S \otimes_R A},\Mod_{S \otimes_R B})^{hG}.
}
\]
The vertical maps are equivalences by the results of \cite[4.8.4]{lurie-higheralgebra}, and the top map is an equivalence by Theorem~\ref{thm:fixedpoint-categories} using the equivalence $S \otimes_R (A \otimes B^\op) \simeq (S \otimes_R A) \otimes_S (S \otimes_R B)^\op$.
\end{proof}

In particular, this gives a descent criterion for Morita equivalence.

\begin{corollary}
\label{cor:moritadetection}
The group $\pi_0 (B\Pic_S^{hG})$ has, as a subgroup, the group of Morita equivalence classes of $R$-algebras $A$ such that there exists an $S$-module $M$ and an equivalence of $S$-algebras $S \otimes_R A \simeq \End_S(M)$.
\end{corollary}

\begin{proof}
The $\i$-category of such $R$-algebras is the preimage of the component of $\Mod_S$ in $\Cat_{S}$, and all such algebras are Azumaya $R$-algebras by the previous result; this component is $B\Pic_S$ by Proposition~\ref{prop:bglquotient}. The maximal subgroupoid in $\Cat_{R}$ spanned by objects in this preimage is therefore equivalent to $(B\Pic_S)^{hG}$, as taking maximal subgroupoids preserves all limits and colimits.
\end{proof}

\subsection{Spectral sequence tools}
\label{sec:spectral-sequences}

For an object $X$ in an $\i$-category $\C$, we write $B\Aut_\C(X)$ for the subgroupoid of $\C^\simeq$ spanned by objects equivalent to $X$ and $\Aut_\C(X)$ for the space of self-equivalences.
\begin{definition}
Let $G$ be a group and $f\co BG \to \Cat_\i$ a functor classifying the action of $G$ on an $\i$-category $\C$. Write $\C_{hG} \to BG$ for the associated fibration (the colimit) and $\C^{hG}$ for the limit.
\end{definition}

Restricting gives us a Kan fibration $(\C_{hG})^\simeq \to BG$ of Kan
complexes whose space of sections is $(\C^{hG})^\simeq$
\cite[3.3.3.2]{lurie-htt}.  The descent diagram in
Theorem~\ref{thm:descentdiag} will now allow us to carry out
computations using the Bousfield--Kan spectral sequence for spaces of
sections. In our cases of interest there will be obstruction groups
that are annihilated by late differentials, and so we need to use the
more sophisticated obstruction theory due to Bousfield
\cite{bousfield-obstructions}. We will review this obstruction theory now.

For a cosimplicial object $\D^\bullet\co \Delta \to \Cat_\i$, the limits
\[
\Tot^n(\D) = \lim_{\Delta_{\leq n}} \D^n
\]
give a tower of $\i$-categories whose limit is the limit of $\D^\bullet$.

\begin{proposition}
Let $f\co \C_{hG} \to BG$ be a Kan fibration classifying the action of $G$ on a Kan complex $\C$, viewed as an $\i$-groupoid. Then there is a tower
\[
\cdots \to \Tot^2 \to \Tot^1 \to \Tot^0 = \C
\]
of Kan fibrations whose limit is $\C^{hG}$. Given an object $X \in \C$, there is an obstruction theory for existence and uniqueness of lifts of $X$ to an object of $\C^{hG}$, natural in $X$ and $\C$.
\begin{enumerate}
\item An object $X \in \C$ is in the essential image of $\Tot^1$ if and only if the equivalence class $[X] \in \pi_0 {\cal C}$ is fixed by the action of the group $G$. Equivalently, this is true if and only if the map
\[
\pi_0 \Aut_{{\cal C}_{hG}} (X) \to \pi_0 \Aut_{BG}(*) = G
\]
is surjective.
\item Given an object $X \in {\cal C}$ with a lift $Y \in \Tot^1$, consider the surjection of groups
\[
\pi_0 \Aut_{{\cal C}_{hG}} (X) \to G
\]
as above. The obstruction to $X$ being in the essential image of $\Tot^2$ is whether this map of groups splits, and the obstruction to uniqueness of lift to $\Tot^2$ is parametrized by the choice of splitting.
\item {\cite[2.4]{bousfield-obstructions}} If $X$ lifts to $Y \in \Tot^n$ for $n \geq 2$, we have a fringed spectral sequence (starting at $E_1$) with $E_2$-term
\[
H^s(G; \pi_t (B\Aut_{\C} X)),
\]
defined for $t > 1$ or for $0 \leq s \leq t \leq 1$. Further pages $E_r^{s,t}$ only exist for $2r-2 \leq n$, and the $E_r$-page depends on a choice of lift of $X$ to $\Tot^{r-1}$. For $r \geq 2$ the $E_r$-page is defined on the region
\[
\{ (s,t) \mid s \geq 0, t-s \geq 0\} \cup \{(s,t) \mid s \geq 0, t-r \geq \tfrac{r-2}{r-1}(s-r)\}.
\]
\item {\cite[5.2]{bousfield-obstructions}} If $r \geq 1$ and $Y$ is a lift of $X$ to $\Tot^r$ which admits a further lift to $\Tot^{2r}$, then there is an obstruction class
\[
\theta_{2r+1} \in E_{r+1}^{2r+1,2r}
\]
which is zero if and only if $Y$ can be lifted to $\Tot^{2r+1}$.
\item {\cite[5.2]{bousfield-obstructions}} If $r \geq 2$ and $Y$ is a lift of $X$ to $\Tot^r$ which admits a further lift to $\Tot^{2r-1}$, then there is an obstruction class
\[
\theta_{2r} \in E_r^{2r,2r-1}
\]
which is zero if and only if $Y$ can be lifted to $\Tot^{2r}$.
\item If $Y \in \C^{hG}$, the above spectral sequences converge (in the region $t-s > 0$) to $\pi_{t-s} (B\Aut_{\C^{hG}}Y)$.
\item If $\C \simeq \Omega \D$ for a Kan complex $\D$ with compatible $G$-action, the spectral sequences for $\C$ and $\D$ are compatible. In particular, if the map $BG \to \S$ representing the $G$-action on $\C$ lifts to a functor from $BG$ to the category of $E_\infty$-spaces, we can construct an associated $K$-theory spectrum $K(\C)$ such that the spectral sequence above extends to the homotopy fixed-point spectral sequence for the action of $G$ on $K(\C)$.
\end{enumerate}
\end{proposition}

\begin{remark}
The beginning of the obstruction theory may be described as follows. In order for $X$ to lift to the limit $\C^{hG}$, a lift to $\Tot^1$ is determined by choosing equivalences $\phi_g\co {}^g X \to X$ for all $g \in G$. A lift to $\Tot^2$ is then determined by witnesses for the cocycle condition, in the form of homotopies from $\phi_{gh}$ to $\phi_g \circ {}^g(\phi_h)$.
\end{remark}

\begin{remark}
The user (particularly if they are used to stable work) may benefit from being explicitly reminded of some of the dangers of the ``fringe effect.'' While the splittings in the second obstruction can be parametrized by $H^1(G;\pi_0 \Aut_{\C}(X))$, this does not occur until an initial splitting is chosen (indeed, otherwise the action of $G$ on $\pi_* \Aut_\C(X)$ is not even defined). The structure of the spectral sequence, at arbitrarily large pages, may also depend strongly on the choices of lift $Y$.
\end{remark}

Because we will be interested in understanding different lifts, it will be useful to be more systematic about the obstructions to this.

\begin{definition}
For an $\i$-category $\C$ and objects $X$ and $Y$ in $\C$, let $\Equiv_{\C}(X,Y)$ be the pullback in the diagram
\[
\xymatrix{
\Equiv_{\C}(X,Y) \ar[r] \ar[d] &
\Map(\Delta^1, \C^\simeq) \ar[d] \\
\{(X,Y)\} \ar[r] &
\Map(\partial \Delta^1, \C^\simeq).
}
\]
\end{definition}

\begin{proposition}
The space $\Equiv_\C(X,Y)$ is a Kan complex, and composition of functions gives a left action of the group $\Aut_\C(Y)$ on $\Equiv_\C(X,Y)$. If $\Equiv_\C(X,Y)$ is nonempty, any choice of point $f\in\Equiv_\C(X,Y)$ produces an equivalence $f^*\co \Aut_\C(Y) \to \Equiv_\C(X,Y)$.
\end{proposition}

\begin{proposition}
Let $G$ be a group acting on an $\i$-category $\C$, let $p\co \C^{hG} \to \C$ be the limit, and suppose and $X$ and $Y$ are objects in $\C^{hG}$. Then the map
\[
\Equiv_{\C^{hG}}(X,Y) \to \Equiv_\C(p(X),p(Y))^{hG}.
\]
is an equivalence of Kan complexes.
\end{proposition}

\begin{proof}
The fixed-point construction, as a limit, commutes with taking maximal subgroupoids, mapping objects, and pullbacks.
\end{proof}

We may therefore apply the tower of $\Tot$-objects to both $\Aut_\C(Y)$ and $\Equiv_\C(X,Y)$ to obtain the following.
\begin{proposition}
Let $f\co \C_{hG} \to BG$ be a Kan fibration classifying the action of $G$ on a Kan complex $\C$, $p\co \C^{hG} \to \C$ the limit, and $X$ and $Y$ objects in $\C^{hG}$.
\begin{enumerate}
\item There are towers of Kan fibrations:
\begin{align*}
  \cdots \to \Aut^2(Y) \to \Aut^1(Y) \to \Aut^0(Y) &= \Aut_\C(p(Y))\\
\!\!\!\!\!\! \cdots \to \Equiv^2(X,Y) \to \Equiv^1(X,Y) \to \Equiv^0(X,Y) &= \Equiv_\C(p(X),p(Y))
\end{align*}
The limits are $\Aut_{\C^{hG}}(Y)$ and $\Equiv_{\C^{hG}}(X,Y)$ respectively.
\item The spaces $\Aut^n(Y)$ are $\i$-groups which act on the spaces $\Equiv^n(X,Y)$.
\item If $\Equiv^n(X,Y)$ is nonempty, any choice of point produces an equivalence of partial towers $\Aut^{\leq n}(Y) \to \Equiv^{\leq n}(X,Y)$.
\item {\cite[5.2]{bousfield-obstructions}} If $\Equiv^{n}(X,Y)$ is nonempty, there is an obstruction class
\[
\theta_{n+1} \in E_r^{n+1,n+1}
\]
in the spectral sequence calculating $\pi_* B\Aut_{\C^{hG}}(Y)$, defined for $2r\leq n+1$, which is zero if and only if $\Equiv^{n+1}(X,Y)$ is nonempty.
\end{enumerate}
\end{proposition}

\section{Calculations}
\label{sec:calculations}

\subsection{Algebraic Brauer groups of even-periodic ring spectra}

In this section we assume that $E$ is an even-periodic $\EE_\i$-ring spectrum: there is a unit in $\pi_2 E$, and $\pi_1 E$ is trivial.

We can describe specific Azumaya algebras for these groups using Theorem~\ref{thm:obstruction-calcs} and the algebras described in the proof of \cite[7.10]{small-brauerwallgroup}. 
\begin{example}
\label{ex:quaternion}
Let $u \in \pi_2 E$ be a unit and $\pi_0 E \to R$ a quadratic Galois extension with Galois automorphism $\sigma$. There is an Azumaya $E$-algebra whose coefficient ring is the graded quaternion algebra
\[
R\langle S \rangle / (S^2 - u, S r - {}^\sigma r S),
\]
where $S$ is in degree $1$ and $R$ is concentrated in degree zero.
\end{example}

\begin{example}
\label{ex:halfquaternion}
Suppose $2$ is a unit in $\pi_0 E$ and $u \in \pi_2 E$ is a unit. There is an Azumaya $E$-algebra whose coefficient ring is (perhaps unexpectedly) the $1$-periodic graded ring
\[
(\pi_* E)[x] / (x^2 - u) \cong (\pi_0 E)[x^{\pm 1}],
\]
which is of rank two over $\pi_* E$. If $A$ and $B$ are two such algebras determined by units $u$ and $v$, then $A \wedge_E B$ is equivalent to a quaternion algebra from Example~\ref{ex:quaternion} determined by the unit $u \in \pi_2(E)$ and the quadratic Galois extension $\pi_0(E) \to \pi_0(E)[y] / (y^2 + uv^{-1})$.
\end{example}

If $E$ is even periodic and we fix a unit $u \in \pi_2 E$, the category of $E$-modules has $\ZZ/2$-graded homotopy groups in the classical sense. Therefore, the set of Morita equivalence classes of algebraic Azumaya algebras over $E$ is the same as the set of Morita equivalence classes of $\ZZ/2$-graded Azumaya algebras over $\pi_0(E)$: the {\em Brauer--Wall} group $BW(\pi_0 E)$. This $\ZZ/2$-graded Brauer group of a commutative ring has been largely determined (generalizing work of Wall over a field). In order to state the result, we will need to recall the definition of the group of $\ZZ/2$-graded quadratic extensions of a ring $R$.
\begin{definition}
Suppose $R$ is a commutative ring, viewed as $\ZZ/2$-graded and concentrated in degree $0$. Then $Q_2(R)$ is the set of isomorphism classes of quadratic graded $R$-algebras: $\ZZ/2$-graded $R$-algebras whose underlying ungraded $R$-algebra is commutative, separable, and projective of rank two. 
\end{definition}
In the ungraded case the corresponding set is identified with the \'etale cohomology group $H^1_{et}(\Spec(R), \ZZ/2)$; similarly $Q_2(R)$ admits a natural group structure. If $\Spec(R)$ is connected, then there are two possible types of element in $Q_2(R)$. In a quadratic graded $R$-algebra $L = (L_0,L_1)$, either $L_1$ has rank $0$ and we have an ungraded quadratic extension $R \to L_0$, or $L_1$ has rank $1$ and $L$ is of the form $(R,L_1)$ for some rank $1$ projective $R$-module $L_1$. In the latter case, the multiplication map $L_1 \otimes_R L_1 \to R$ must be an isomorphism. Carrying this analysis further yields the following result.
\begin{proposition}
When $\Spec(R)$ is connected, there is a short exact sequence
\[
0 \to H^1_{et}(R, \ZZ/2) \to Q_2(R) \to \ZZ/2.
\]
Here the \'etale cohomology group $H^1_{et}(R, \ZZ/2)$ parametrizes
ungraded $\ZZ/2$-Galois extensions of $R$, and the map $Q_2(R) \to \ZZ/2$ sends a $\ZZ/2$-graded quadratic $R$-algebra $(L_0, L_1)$ to the rank of $L_1$. The image of $Q_2(R)$ in $\ZZ/2$ is nontrivial if and only if $2$ is a unit in $R$.
\end{proposition}
\begin{theorem}[{\cite{small-brauerwallgroup}}]
\label{thm:brauerwall-sequence}
Suppose that $R$ possesses no idempotents. Then the Brauer--Wall group $BW(R)$ is contained in a short exact sequence
\[
0 \to \Br(R) \to BW(R) \to Q_2(R) \to 0,
\]
where the subgroup is generated by Azumaya algebras concentrated in even degree.
\end{theorem}
\begin{corollary}
Suppose that $E$ is even-periodic and that $\pi_0 E$ possesses no idempotents.
Then the subgroup of the Brauer group of $E$ generated by algebraic Azumaya algebras is contained in a short exact sequence
\[
0 \to \Br(\pi_0 E) \to \pi_0 \Br(E)^{\rm alg} \to Q_2(\pi_0 E) \to 0,
\]
where the subgroup is generated by algebraic Azumaya algebras with homotopy concentrated in even degrees. In $Q_2(\pi_0 E)$, the elements of $H^1_{et}(\pi_0 E, \ZZ/2)$ detect the algebras of Example~\ref{ex:quaternion}, while the map to $\ZZ/2$ detects any of the ``half-quaternion'' algebras of Example~\ref{ex:halfquaternion}.
\end{corollary}

\begin{example}
\label{ex:KU-algebraic}
In the case where $E$ is the complex $K$-theory spectrum $KU$, with coefficient ring $\ZZ[\beta^{\pm 1}]$, the relevant Brauer--Wall group $BW(\ZZ)$ is trivial and all $\ZZ/2$-graded algebraic Azumaya algebras are Morita equivalent. Therefore, there are no $\ZZ$-graded algebraic Azumaya algebras over $KU$ other than those of the form $\End_{KU} (M)$ for $M$ a coproduct of suspensions of $KU$.
\end{example}
\begin{example}
Suppose that $\pi_0 E$ is a Henselian local ring with residue field $k$. Then extension of scalars determines isomorphisms $H^1_{et}(\pi_0 E, \ZZ/2) \to H^1_{et}(k, \ZZ/2)$ and $\Br(\pi_0 E) \to \Br(k)$ (\cite[5]{azumaya-algebras}, \cite[6.1]{grothendieck-brauerI}), and hence an isomorphism $BW(\pi_0 E) \to BW(k)$. If $k$ is finite (for example, when $E$ is a Lubin--Tate spectrum associated to a formal group law over a finite field) the group $\Br(k)$ is trivial and the Galois cohomology group is $\ZZ/2$, so we find that the Brauer--Wall group of $k$ is $\ZZ/2$ if $k$ has characteristic $2$ and is of order $4$ if $k$ has odd characteristic. The algebraic $\ZZ/2$-graded Azumaya $E$-algebras are generated (up to Morita equivalence) by those of Examples~\ref{ex:quaternion} and~\ref{ex:halfquaternion}.
\end{example}
\begin{example}
If we form the localized ring $KU[1/2]$, we may use global class field theory to analyze the result. The ordinary Brauer group is $\ZZ/2$, generated by the Hamilton quaternions over $\ZZ[1/2]$, and this algebra lifts to an Azumaya algebra as originally shown in \cite[6.3]{baker-richter-szymik}. The \'etale cohomology group is $\ZZ/2 \times \ZZ/2$, with nonzero elements corresponding to the quadratic extensions obtained by adjoining $i$, $\sqrt{2}$, or $\sqrt{-2}$. Finally, $KU[1/2]$ also has Azumaya algebras given by its $1$-periodifications, generating the quotient $\ZZ/2$ of the Brauer--Wall group $BW(\ZZ[1/2])$. The full group has order $16$, and one can show that it is isomorphic to $\ZZ/8 \times \ZZ/2$. These can be given specific generators: the $\ZZ/8$-factor is generated by an algebra with coefficient ring $\ZZ[\beta^{\pm(1/2)}, 1/2]$ as an algebra over $KU_*$, while the $\ZZ/2$-factor is generated by an algebra with coefficient ring
\[
KU_*\left[\sqrt 2, 1/2\right]\langle S\rangle / (S^2 - \beta, S \sqrt 2 + \sqrt 2 S).
\]
\end{example}

\begin{remark}
The short exact sequence of Theorem~\ref{thm:brauerwall-sequence} is generalized in \cite[Section~4]{childs-garfinkel-orzech} for many more groups, and by applying their results one can compute the Brauer--Wall group classifying algebraic Azumaya algebras for an overwhelming abundance of examples. For the $4$-periodic localization $KO[1/2]$ we may show that the Brauer--Wall group has 16 elements, combining the order-2 Brauer group of $\ZZ[1/2]$ with the order-8 collection of Galois extensions of $\ZZ[1/2]$ with cyclic Galois group of order four. For the $p$-complete Adams summand $L_p$ at an odd prime $p$, the Brauer--Wall group has $(p-1)$ elements if $p \equiv 1 \mod 4$ and $2(p-1)$ elements if $p \equiv 3 \mod 4$. By contrast, $p$-{\em local} spectra such as $K_{(p)}$, $KO_{(p)}$, or $L_{(p)}$ tend to have much larger Brauer groups because $\ZZ_{(p)}$ and its finite extensions have infinite Brauer groups.
\end{remark}

\subsection{Homotopy fixed-points of $\Pic(KU)$}

In this section we study the Galois extension $KO \to KU$. Most of the structure of the homotopy fixed-point spectral sequence for $\Pic(KU)$ has been determined in depth by Mathew and Stojanoska using tools they developed for comparing with the homotopy fixed-point spectral sequence for $KU$ \cite[7.1]{mathew-stojanoska-picard}. However, for our purposes we will require information about the behavior of the spectral sequence in small, negative degrees.

We recall the following about the category of naive $G$-spectra.

\begin{proposition}
For a $G$-equivariant spectrum $X$ such that $\pi_i(X) = 0$ for $n < i < m$, the $d_{n-m+1}$-differential
\[
H^s(G;\pi_n(X)) \to H^{s+m-n+1}(G;\pi_m(X))
\]
in the homotopy fixed-point spectral sequence for $X^{hG}$ is given by an equivariant $k$-invariant
\[
k^G \in \pi_{n-m-1} F_{\SS[G]}(H\pi_n X, H\pi_m X),
\]
which determines a cohomology operation of degree $(m-n+1)$ on Borel equivariant cohomology. The forgetful map
\[
\pi_{n-m-1} F_{\SS[G]}(H\pi_n X, H\pi_m X) \to \pi_{n-m-1} F_{\SS}(H\pi_n X, H\pi_m X)
\]
sends $k^G$ to the underlying $k$-invariant of $X$.
\end{proposition}

Using the adjunction
\[
F_{\SS[G]}(X,Y) \simeq F_{\SS[G]}(\SS,F_{\SS}(X,Y)) = F_\SS(X,Y)^{hG},
\]
we recover the following computational tool.

\begin{proposition}
\label{prop:cohomops}
For functors $BG \to \S$ representing spectra $X$ and $Y$ with $G$-action, there exists a spectral sequence with $E_2$-term
\[
E_2^{s,t} = H^s(G; \pi_t F_\SS(X,Y)) \Rightarrow \pi_{t-s} F_{\SS[G]}(X,Y).
\]
Furthermore, the edge morphism in this spectral sequence recovers the natural map to $\pi_* F_\SS(X,Y)$.
\end{proposition}

We may then apply this to calculate the possible first two $C_2$-equivariant $k$-invariants of $\pic(KU)$, both between degrees $0$ and $1$ and between degrees $1$ and $3$.

\begin{proposition}
Let $\ZZ^-$ be $\ZZ$ with the sign action of $C_2$, and
\[
\beta^-\co H^*(C_2;\ZZ/2) \to H^{*+1}(C_2;\ZZ^-)
\]
the Bockstein map associated to the short exact sequence
\[
0 \to \ZZ^- \to \ZZ^- \to \ZZ/2 \to 0.
\]
Let $x \in H^1(G;\ZZ/2)$ denote the generator. We have
\[
\pi_{-2} F_{\SS[C_2]} (H\ZZ/2, H\ZZ/2) \cong (\ZZ/2)^3,
\]
generated by the operations $\Sq^2(-)$, $x \cdot \Sq^1(-)$, and $x^2 \cdot (-)$. We also have
\[
\pi_{-3} F_{\SS[C_2]} (H\ZZ/2, H\ZZ^-) \cong (\ZZ/2)^2,
\]
generated by the operations $\beta^- \circ \Sq^2(-)$ and $\beta^-(x^2 \cdot (-))$.

The restriction to the group of nonequivariant operations sends the generators involving $x$ to zero.
\end{proposition}

\begin{proof}
Proposition~\ref{prop:cohomops} gives us two spectral sequences, pictured in Figure~\ref{fig:kinvt}:
\begin{align*}
  H^s(C_2; \pi_t F_\SS(H\ZZ/2, H\ZZ/2)) &\Rightarrow \pi_{t-s} F_{\SS[C_2]}(H\ZZ/2, H\ZZ/2)\\
  H^s(C_2; \pi_{t-s} F_\SS(H\ZZ/2, H\ZZ)) &\Rightarrow \pi_{t-s} F_{\SS[C_2]}(H\ZZ/2, H\ZZ^-)
\end{align*}
There is an isomorphism $\pi_{-*} F_\SS(H\ZZ/2, H\ZZ/2) \cong A^*$, where $A^*$ is the mod-2 Steenrod algebra; this group is isomorphic to $\ZZ/2$ for $-2 \leq * \leq 0$ and is trivial for all other $* \geq -2$. Similarly, there is an isomorphism $\pi_{-*} F_\SS(H\ZZ/2, H\ZZ) \cong \Sq^1 \cdot A^* \subset A^*$; this group is isomorphic to $\ZZ/2$ for $* = -1, -3$ and is trivial for all other $* \geq -3$. The associated spectral sequences appear in Figure~\ref{fig:kinvt}.
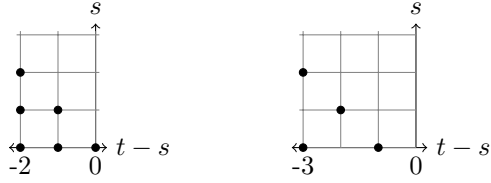
\begin{figure}
\centering
\begin{tikzpicture}[scale=0.5]
\draw[->] (0,0) -- (0,3.3);
\node[above] at (0,3.3) {$s$};
\node[right] at (0.3,0) {$t-s$};
\node[below] at (0,0) {0};
\node[below] at (-2,0) {-2};
\draw[<->] (-2.3,0) -- (0.3,0);
\draw[help lines] (-2.1,0) grid (0.1,3.1);
\filldraw (-2,2) \tdot (-2,1) \tdot (-2,0) \tdot (-1,1) \tdot (-1,0) \tdot (0,0) \tdot;
\end{tikzpicture}
\hskip 3pc
\begin{tikzpicture}[scale=0.5]
\draw[->] (0,0) -- (0,3.3);
\node[above] at (0,3.3) {$s$};
\node[right] at (0.3,0) {$t-s$};
\node[below] at (0,0) {0};
\node[below] at (-3,0) {-3};
\draw[<->] (-3.3,0) -- (0.3,0);
\draw[help lines] (-3.1,0) grid (0,3.1);
\filldraw (-3,2) \tdot (-3,0) \tdot (-2,1) \tdot (-1,0) \tdot;
\end{tikzpicture}
\caption{Spectral sequences for equivariant $k$-invariants}
\label{fig:kinvt}
\end{figure}
These spectral sequences place an upper bound of $8$ on the size of the group $\pi_{-2} F_{\SS[C_2]}(H\ZZ/2, H\ZZ/2)$ and of $4$ on the size of the group $\pi_{-3} F_{\SS[C_2]}(H\ZZ/2, H\ZZ^-)$. However, these cohomology operations we have described in these groups are linearly independent over $\ZZ/2$, as can be checked by applying them to elements in the group 
\[
\pi_* F_{\SS[C_2]}(\Sigma^\infty_+ EC_2, H\ZZ/2) \cong H^*(BC_2; \ZZ/2).
\]
(These represent elements in different cohomological filtration in this spectral sequence.)
\end{proof}

\begin{proposition}
The first two $C_2$-equivariant $k$-invariants of $\pic(KU)$ are $\Sq^2 + x\Sq^1$ and $\beta^- \Sq^2$.
\end{proposition}

\begin{proof}
The underlying nonequivariant $k$-invariants must be the first two $k$-invariants of $\pic(KU)$. These are $\Sq^2$ and $\beta \Sq^2$, where $\beta$ is the nonequivariant Bockstein.

Moreover, the generating elements in $\pi_0 \pic(KU)$ and $\pi_1 \pic(KU)$ are the images of the classes $[\Sigma KO]$ and ${-1}$ from $\pic(KO)$ respectively, and hence must survive the homotopy fixed-point spectral sequence. These classes would support a nontrivial $d_2$ or $d_3$ differential if the cohomology operation involved a nonzero multiple of $x^2$ or $\beta^- x^2$ respectively. This shows that the second $k$-invariant can only be $\beta^- \Sq^2$, and the first $k$-invariant can only be $\Sq^2$ or $\Sq^2 + x\Sq^1$.

Suppose that the second $k$-invariant were $\Sq^2$. This $k$-invariant is in the image of the map
\[
\pi_{-2} F_{\SS}(H\ZZ/2,H\ZZ/2) \to \pi_{-2} F_{\SS[C_2]}(H\ZZ/2,H\ZZ/2)
\]
induced by the ring map $\SS[C_2] \to \SS$, and so the resulting
$C_2$-equivariant Postnikov stage $\tau_{\leq 1} \pic(KU)$ would be
equivalent to one with the trivial $C_2$-action. We would then have
the equivalence
\[
(\tau_{\leq 1} \pic(KU))^{hC_2} \simeq F((BC_2)_+, \tau_{\leq 1} \pic(KU)).
\]
This splits off a copy of $\tau_{\leq 1} \pic(KU)$ so there could be no hidden extensions from $H^0(C_2; \pi_0
\pic(KU)$ to $H^1(C_2; \pi_1 \pic(KU))$ in the homotopy fixed-point
spectral sequence. However, there is a hidden extension:
the class $[\Sigma KO] \in \pi_0 \pic(KO)$ has nontrivial image in $H^0(C_2; \pi_0 \pic(KU))$ and twice it is $[\Sigma^2 KO]$, which has nontrivial image in $H^1(C_2; (\pi_0 KU)^\times)$.
\end{proof}

\begin{proposition}
\label{prop:fixedpointcalc}
The homotopy fixed-point space $B\Pic(KU)^{hC_2}$ has homotopy groups
\[
\pi_n B\Pic(KU)^{hC_2} =
\begin{cases}
\pi_{n-2} GL_1(KO)&\text{if }n \geq 2\\
\ZZ/8&\text{if }n = 1\\
\ZZ/2&\text{if }n = 0.
\end{cases}
\]
\end{proposition}

\begin{proof}
The homotopy fixed-point spectral sequence
\[
H^s(C_2; \pi_t \pic(KU)) \Rightarrow \pi_{t-s} \pic(KU)^{hC_2}
\]
is pictured in Figure~\ref{fig:stable}; we refer to \cite{mathew-stojanoska-picard} for the portion with $t > 3$, obtained by comparison with the homotopy fixed point spectral sequence for $KU$. The differentials supported on $t=0$ and $t=1$ are the stable cohomology operations we just determined.
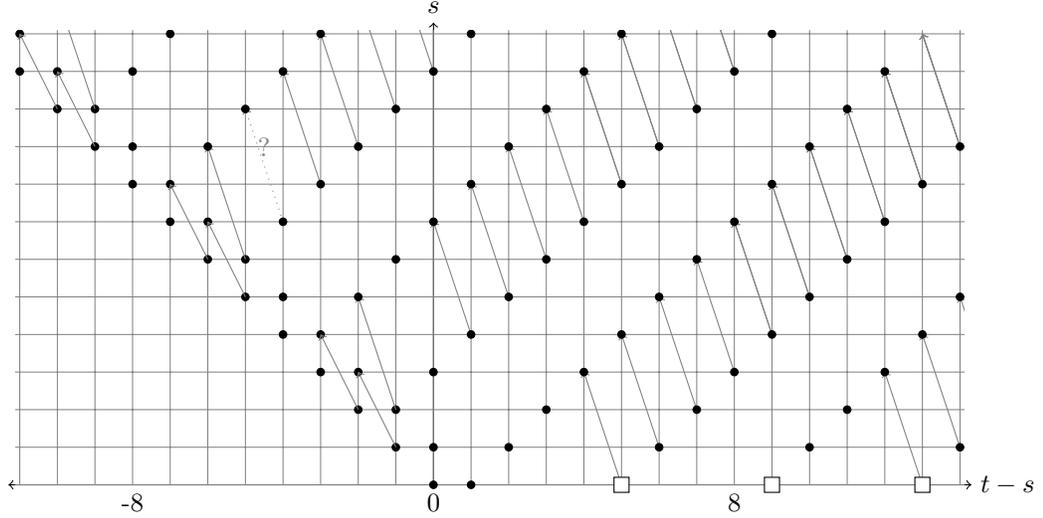
\begin{figure}
\centering
\begin{tikzpicture}[scale=0.5]
\draw[->] (0,0) -- (0,12.3);
\node[above] at (0,12.3) {$s$};
\node[right] at (14.3,0) {$t-s$};
\node[below] at (0,0) {0};
\node[below] at (8,0) {8};
\node[below] at (-8,0) {-8};
\draw[<->] (-11.3,0) -- (14.3,0);
\clip (-11.1,-1) rectangle (14.1,12.1);
\draw[help lines] (-12.1,0) grid (20.1,12.1);

\filldraw \foreach \x in {0,...,12} {(-\x,\x) \tdot};
\filldraw \foreach \x in {0,...,12} {(-\x+1,\x) \tdot};

\foreach \x in {5,13,...,24} {
  \foreach \y in {0,...,12} {
    \foreach \d in {0,4,...,12} {
      \draw[->,color=gray] (\x+\y-\d,\y+\d) -- (\x+\y-\d-1,\y+\d+3);}}}

\foreach \x in {0,4,...,12} {
  \draw[->,color=gray] (-\x-1,\x+1) -- (-\x-2,\x+3);
  \draw[->,color=gray] (-\x-2,\x+2) -- (-\x-3,\x+4);
  \draw[->,color=gray] (-\x-1,\x+2) -- (-\x-2,\x+5);
}

\draw[->,dotted,color=gray] (-4, 7) -- (-5,10);
\node[above,color=gray] at (-4.5,8.5) {?};

\filldraw[fill=white] \foreach \x in {5,9,...,16} {(\x-.2,-.2)
  rectangle (\x+.2,.2)};
\filldraw \foreach \x in {3,7,...,24} {\foreach \y in {1,3,...,12}
  {(\x-\y,\y) \tdot}};
\filldraw \foreach \x in {5,9,...,24} {\foreach \y in {2,4,...,12}
  {(\x-\y,\y) \tdot}};
\end{tikzpicture}
\caption{Fixed-point spectral sequence for $\Pic(KU)$ up to $E_3$}
\label{fig:stable}
\end{figure}
The inclusion of $\ZZ/8$ into $\pi_0 \pic(KO) \cong \pi_0 \pic(KU)^{hC_2}$ forces the hidden extension in degree $0$.
\end{proof}

This recovers the calculation of the Picard group of $KO$ by \cite{hopkins-mahowald-sadofsky}.

There are potential further differentials in negative degrees in the homotopy fixed-point spectral sequence which we have not addressed here. There are potential sources for a $d_4$-differential when $t=0$, $s \equiv 3 \mod 4$. There are also potential targets for a $d_3$- or $d_5$- or $d_6$-differential when $t=5$, $s \equiv 2 \mod 4$, though these latter would be impossible if the Postnikov stage $\pic(KU) \to \tau_{\leq 3} \pic(KU)$ split off equivariantly. It seems likely that a precise formulation of the periodic structure in this spectral sequence would be able to address these questions.

\subsection{Lifting from $KU$ to $KO$}

In this section we examine those Azumaya $KO$-algebras whose extension to $KU$ are algebraic.

By Example~\ref{ex:KU-algebraic}, we have the following.
\begin{proposition}
Any algebraic Azumaya $KU$-algebra is of the form $\End_{KU} N$, where $N$ is a finite coproduct of suspensions of $KU$.
\end{proposition}

Therefore, by Proposition~\ref{prop:fixedpointcalc} and Corollary~\ref{cor:moritadetection}, there are at most two Morita equivalence classes of Azumaya $KO$-algebras whose extensions to $KU$ are algebraic.

The following shows that the nontrivial Morita equivalence class is realizable.

\begin{proposition}
There exists a unique equivalence class of quaternion algebra $Q$ over $KO$ such that
\begin{itemize}
\item $KU \otimes_{KO} Q \simeq M_2(KU)$, and
\item there is no $KO$-module $M$ such that $Q \not\simeq \End_{KO}(M)$ as $KO$-algebras.
\end{itemize}
This algebra has homotopy groups isomorphic, as a $KO_*$-algebra, to the homotopy groups of a twisted group algebra:
\[
\pi_* Q \cong \pi_* KU\langle C_2 \rangle \cong \pi_* \End_{KO} KU.
\]
\end{proposition}

\begin{proof}
The $KO$-algebras $A$ such that $KU \otimes_{KO} A \simeq M_2(KU)$ are parametrized by the preimage of the component $B\Aut_{\Alg_{KU}} M_2(KU) \subset \Az_{KU}$. We may therefore apply the obstruction theory of Section~\ref{sec:spectral-sequences}. We know that there is a chain of equivalences
\begin{align*}
KU \otimes_{KO} \End_{KO}(KU) &\simeq 
\End_{KU}(KU \otimes_{KO} KU) \\&\simeq 
\End_{KU}(KU \oplus KU) \simeq 
M_2(KU),
\end{align*}
and so we may use $\End_{KO}(KU)$ as a basepoint for the purposes of calculations. The obstruction theory then takes place in a fringed spectral sequence with $E_2$-term
\[
H^s(C_2; \pi_t B\Aut_{\Alg_{KU}}(M_2(KU))).
\]

By Corollary~\ref{cor:pgl-sequence}, we have a long exact sequence
\begin{align*}
  \cdots &\to (\pi_n KU)^\times \to (\pi_n M_2(KU))^\times \to \pi_n (\Aut_{\Alg_{KU}}(M_2 KU)) \to \cdots\\
  &\to (\pi_0 KU)^\times \to (\pi_0 M_2(KU))^\times \to \pi_0
  (\Aut_{\Alg_{KU}}(M_2(KU))) \to \pi_0 \Pic(KU).
\end{align*}
Since $\pi_* M_2(KU) \cong M_2 (\pi_*(KU))$, we find that $\Aut_{\Alg_{KU}}(M_2 KU)$ has trivial homotopy groups in odd degrees, and that for $k > 0$ there are short exact sequences
\[
0 \to \pi_{2k} KU \to \pi_{2k} M_2(KU) \to \pi_{2k} \Aut_{\Alg_{KU}}(M_2 KU) \to 0.
\]
Moveover, the $C_2$-action on $\pi_{2k}(KU \otimes_{KO} \End_{KO}(KU)) \cong M_2(KU_{2k})$ is given in matrix form by
\[
\begin{bmatrix}
a & b \\ c & d
\end{bmatrix}
\mapsto
(-1)^k \begin{bmatrix}
d & c \\ b & a
\end{bmatrix}.
\]
We may now use this to calculate group cohomology. We find that for $s,t > 0$, the cohomology $H^s(C_2; \pi_t M_2(KU))$ vanishes with this action and we have isomorphisms
\[
H^s(C_2; \pi_t \Aut_{\Alg_{KU}}(M_2 KU)) \to H^{s+1}(C_2; \pi_t \Pic(KU)),
\]
realized by the natural map $B\Aut_{\Alg_{KU}}(M_2 KU) \to B\Pic(KU)$. We display the spectral sequence for calculating lifts of $M_2(KU)$ in Figure~\ref{fig:unstable} through the $E_3$-term.  The regions where the spectral sequence is undefined at $E_2$ or $E_3$ are blocked out, and the nonabelian cohomology $H^s(C_2; \PGL_2(\ZZ))$ is indicated with $\circledast$.
\begin{figure}
\centering
\begin{tikzpicture}[scale=0.5,>=stealth]
\draw[->] (0,0) -- (0,12.3);
\node[above] at (0,12.3) {$s$};
\node[right] at (14.3,0) {$t-s$};
\node[below] at (0,0) {0};
\node[below] at (8,0) {8};
\node[below] at (-8,0) {-8};
\draw[<->] (-11.3,0) -- (14.3,0);
\clip (-11.1,-1) rectangle (14.1,12.1);
\draw[help lines] (-12.1,0) grid (20.1,12.1);
\draw[pattern=crosshatch dots,pattern color=black!40] (0,0) -- (0,2) -- (-11,13) -- (-13,13) -- (-13,0);
\draw[pattern=crosshatch dots,pattern color=black!20] (0,2) -- (0,3) -- (-5,13) -- (-11,13);

\foreach \x in {7,15,...,24} {
  \foreach \y in {0,...,12} {
      \draw[->,color=gray] (\x+\y,\y) -- (\x+\y-1,\y+3);}}

\foreach \x in {2,-2} {
  \foreach \y in {0,...,12} {
      \draw[->,color=gray] (\x+\y,5-\x+\y) -- (\x+\y-1,8-\x+\y);}}

\foreach \x in {3,7,...,16} {
  \filldraw[fill=white] (\x-.2,-.2) rectangle (\x+.2,.2);
  \filldraw[fill=white] (\x-.1,-.1) rectangle (\x+.1,.1);}
\foreach \x in {5,9,...,16} {
  \filldraw[fill=white] (\x-.2,-.2) rectangle (\x+.2,.2);}
\filldraw \foreach \x in {3,7,...,24} {\foreach \y in {2,4,...,12}
  {(\x-\y,\y) \tdot}};
\filldraw \foreach \x in {5,9,...,24} {\foreach \y in {1,3,...,12}
  {(\x-\y,\y) \tdot}};
\filldraw[fill=white] (1,0) circle [radius=0.3] (0,1) circle [radius=0.3];
\node at (1,0) {$\ast$};
\node at (0,1) {$\ast$};
\end{tikzpicture}
\caption{Fixed-point spectral sequence for $B\Aut(KU \otimes_{KO} \End_{KO}(KU))$ up to $E_3$}
\label{fig:unstable}
\end{figure}
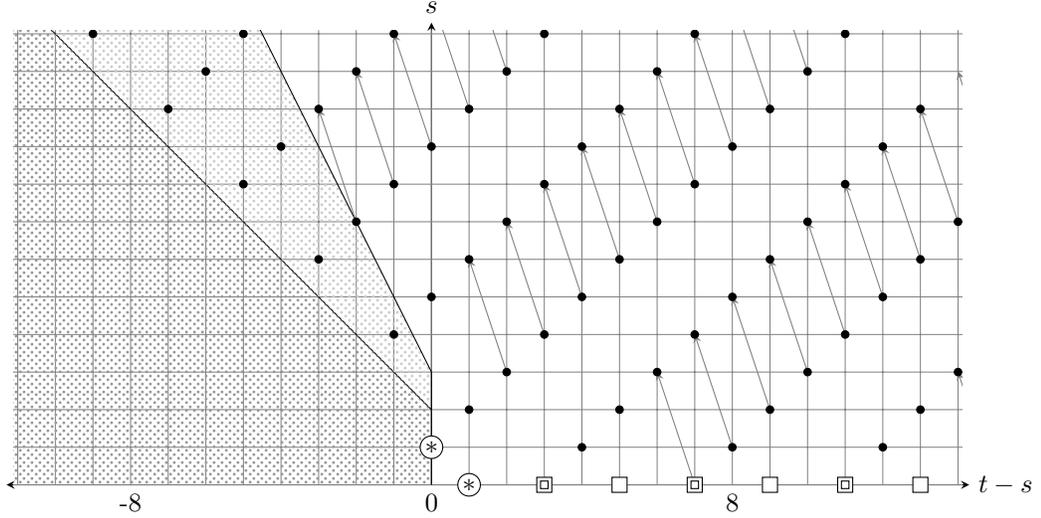
The first detail we note about this spectral sequence is that for $t-s \geq -1$, the $E_4$-page vanishes entirely for $s \geq 5$.
There are potential obstructions to lifting in the column $t-s = -1$ and to uniqueness in the column $t-s = 0$; we will now discuss these obstruction groups using the machinery of Section~\ref{sec:spectral-sequences}.

Because the groups $E_r^{s,s-1}$ and $E_r^{s,s}$ become trivial at
$E_4$ for $s > 5$, there are no obstructions to existence
or uniqueness of lifting algebras beyond $\Tot^5$: any Azumaya $KU$-algebra equivalent to $M_2(KU)$ with a lift to $\Tot^5$ has an essentially unique further lift to an Azumaya $KO$-algebra.

The group $E_2^{4,3}$ is $\ZZ/2$, and this group is a potential home for obstructions for a point in $\Tot^2$ which lifts to $\Tot^3$ to also lift to $\Tot^4$ (see Remark~\ref{rmk:quaternion-obstruction} for further elaboration). Since we have already chosen a lift of $M_2(KU)$ to the algebra $\End_{KO}(KU)$ in the homotopy limit to govern the obstruction theory, the obstruction must be zero at this basepoint.

The group $E_3^{5,5}$ parametrizes differences between lifts from $\Tot^4$ to $\Tot^5$. This group is $\ZZ/2$, and contains only permanent cycles due to the fact that the spectral sequence has a vanishing region at $E_4$. Therefore, there are two distinct lifts of $KU \otimes_{KO} \End_{KO}(KU)$ from $\Tot^2$ to $\Tot^5$, representing two inequivalent $KO$-algebras which become equivalent to $M_2(KU)$ after extending scalars. One of these is $\End_{KO}(KU)$; we will refer to the other algebra as $Q$.

Moreover, the map $B\Aut_{\Alg_{KU}}(M_2 KU) \to B\Pic(KU)$ induces an isomorphism on homotopy fixed-point spectral sequences in the relevant degree. The generator of $E_3^{5,5}$ representing $Q$ therefore maps to the nontrivial element of $\pi_0 (B\Pic(KU))^{hC_2} \subset \pi_0 \Br(KO)$, and so any points of the fixed-point category with distinct lifts to $\Tot^5$ are Morita inequivalent.

Hence, there exists precisely one other $KO$-algebra, $Q$, whose image in $\Tot^4$ is the same as the image of $\End_{KO}(KU)$, and $Q$ is Morita inequivalent to any endomorphism algebra.
\end{proof}

\begin{remark}
\label{rmk:quaternion-obstruction}
The obstruction group $E_2^{4,3}$ deserves some mention. There is an element in $H^1(C_2; \pi_1 B\Aut_{\Alg_{KU}}(M_2 KU))$ whose image in $H^2(C_2; \pi_2 B\Pic(KU))$ is nontrivial. More explicitly, $\pi_1 B\Aut_{\Alg_{KU}}(M_2 KU)$ contains $PGL_2(\ZZ)$ and this $H^1$-class is represented by the alternative action
\[
\begin{bmatrix}
a & b \\ c & d
\end{bmatrix}
\mapsto
(-1)^t
\begin{bmatrix}
0 & 1 \\ -1 & 0
\end{bmatrix}
\begin{bmatrix}
a & b \\ c & d
\end{bmatrix}
\begin{bmatrix}
0 & -1 \\ 1 & 0
\end{bmatrix}
\]
of $C_2$ on $\pi_{2t} M_2(KU)$. One might hope that there is a $KO$-algebra $A$ such that $KU \otimes_{KO} A$ is $M_2(KU)$ with this alternative $C_2$-action on the coefficient ring.

For example, we might imagine finding a self-map $\phi\co KO \to KO$ representing multiplication by $-1 \in \pi_0(KO)$, and using it to produce an action of $C_2$ on $KU \otimes_{KO} M_2(KO)$ such that the generator acts on the $KU$ factor by complex conjugation and on the $M_2(KO)$-factor by
\[
\begin{bmatrix}
a & b \\ c & d
\end{bmatrix}
\mapsto
\begin{bmatrix}
0 & 1 \\ \phi & 0
\end{bmatrix}
\begin{bmatrix}
a & b \\ c & d
\end{bmatrix}
\begin{bmatrix}
0 & \phi \\ 1 & 0
\end{bmatrix}.
\]

In the classical setup, one encounters a sequence of difficulties with carrying this program out. The spectrum $KO$ cannot be a fibrant-cofibrant $KO$-module if $KO$ is strictly commutative, so we require a replacement in order for $\phi$ to be defined. Then this replacement is not strictly the unit for the smash product and so we cannot move $\phi$ across a smash product without an intervening homotopy. In order to make this a ring homomorphism one either wants $\phi^2$ to be the identity, or one wants to replace $\phi$ by an automorphism so that we can genuinely replace this with a conjugation action. And so on. One is left with the feeling that these are technical details and the tools are just barely inadequate for the job, but this is not the case: this $H^1$-class cannot be realized by an algebra {\em at all} because the image in $H^2(C_2;\pi_2 B\Pic(KU))$ supports a $d_3$-differential (see Figure~\ref{fig:stable}). These seemingly mild details are {\em fundamental} to the situation.
\end{remark}


\bibliography{pgl}

\end{document}